\documentclass[11pt]{article}
\usepackage{amsfonts}
\usepackage{amsmath,amssymb}
\usepackage{amsthm}
\usepackage[mathscr]{euscript}
\usepackage{mathrsfs}
\usepackage{indentfirst}
\usepackage{color}
\usepackage{accents}
\usepackage{enumerate}

\newtheorem{theo}{Theorem}[section]
\newtheorem{lem}{Lemma}[section]

\newtheorem{remark}{Remark}[section]
\newtheorem{prop}{Proposition}[section]
\numberwithin{equation}{section}

\allowdisplaybreaks

\newcommand{\be}{\begin{equation}}
\newcommand{\ee}{\end{equation}}
\newcommand\bes{\begin{eqnarray}} \newcommand\ees{\end{eqnarray}}
\newcommand{\bess}{\begin{eqnarray*}}
\newcommand{\eess}{\end{eqnarray*}}
\newcommand{\bbbb}{\left\{\begin{aligned}}
\newcommand{\nnnn}{\end{aligned}\right.}
\newcommand{\bea}{\begin{align*}}
\newcommand{\eea}{\end{align*}}

\newcommand\ep{\varepsilon}

\newcommand\dd{\displaystyle}

\newcommand\dy{{\rm d}y}

\newcommand\yy{\infty}

\newcommand\ud{\underline}

\begin{document}\thispagestyle{empty}
\begin{center}
 {\LARGE\bf Dynamics for a diffusive epidemic model with a free boundary: sharp asymptotic profile\footnote{This work was supported by NSFC Grants
12171120, 11901541,12301247}}\\[4mm]
{\Large Xueping Li}\\[0.5mm]
{School of Mathematics and Information Science, Zhengzhou University of Light Industry, Zhengzhou, 450002, China}\\[2.5mm]
  {\Large Lei Li}\\[0.5mm]
{College of Science, Henan University of Technology, Zhengzhou, 450001, China}\\[2.5mm]
{\Large Mingxin Wang\footnote{Corresponding author. {\sl E-mail}: mxwang@hpu.edu.cn}}\\[0.5mm]
 {School of Mathematics and Information Science, Henan Polytechnic University, Jiaozuo, 454003, China}
\end{center}

\date{\today}

\begin{quote}
\noindent{\bf Abstract.}This paper concerns the sharp asymptotic profiles of the solution of a diffusive  epidemic model with one free boundary and one fixed boundary which is subject to the homogeneous Dirichlet boundary condition and Neumann boundary condition, respectively. The longtime behaviors has been proved to be governed by a spreading-vanishing dichotomy in \cite{LL}, and when spreading happens, the spreading speed is determined in \cite{LLW}. In this paper, by constructing some subtle upper and lower solutions, as well as employing some detailed analysis, we improve the results in \cite{LLW} and obtain the sharp asymptotic spreading profiles, which show the homogeneous Dirichlet boundary condition and Neumann boundary condition imposed at the fixed boundary $x=0$ lead to the same asymptotic behaviors of $h(t)$ and $(u,v)$ near the spreading front $h(t)$.


\end{quote}

\pagestyle{myheadings}
\section{Introduction}
\renewcommand{\thethm}{\Alph{thm}}
{\setlength\arraycolsep{2pt}

Using reaction-diffusion equation (system) to model the spread of epidemics is a hot topic in biomathematics. Relevant studies not only reveal some interesting propagation phenomena, but also promote the development of corresponding mathematical theories. To study the spread of oral-faecal transmitted epidemics, Hsu and Yang \cite{HY} proposed a reaction-diffusion system
 \bes\label{1.1}
\left\{\!\begin{aligned}
&u_{t}=d_1 u_{xx}-au+H(v), & &t>0,~x\in\mathbb{R},\\
&v_{t}=d_2 v_{xx}-bv+G(u), & &t>0,~x\in\mathbb{R},
\end{aligned}\right.
 \ees
 where $H(v)$ and $G(u)$ satisfy
\begin{enumerate}
\item[{\bf(H)}]\; $H,G\in C^2([0,\yy))$, $H(0)=G(0)=0$, $H'(z),G'(z)>0$ in $[0,\yy)$, $H''(z), G''(z)<0$ in $(0,\yy)$, and $G(H(\hat z)/a)<b\hat{z}$ for some $\hat{z}>0$.
 \end{enumerate}

In this model, $u(t,x)$ and $v(t,x)$ stand for the spatial
concentrations of bacteria and infective human population, respectively, at time $t$ and location $x$ in the one dimensional habitat; $-au$ and $H(v)$ represent the natural death rate of bacterial population and the contribution of infective human to the growth rate of the bacteria, respectively; $-bv$ and $G(u)$ are the fatality rate of infective human population and the infection rates of human population, respectively; $d_1$ and $d_2$, respectively, stand for the diffusion rate of bacteria and infective human. Hsu and Yang in \cite{HY} showed that when
$ \mathcal{R}_0:=\frac{H'(0)G'(0)}{ab}>1$,
there exists a $c_*>0$, which is determined by the characteristic polynomial of the linearized system of \eqref{1.1} at $(0,0)$, such that if and only if $c\ge c_*$, \eqref{1.1} has a monotone travelling wave solution $(\phi_1,\phi_2)$ which is unique up to translation and satisfies
 \bess
\left\{\!\begin{aligned}
&d_1\phi''_1-c\phi'_1-a\phi_1+H(\phi_2)=0, \;\;x\in\mathbb{R},\\
&d_2\phi''_2-c\phi'_2-b\phi_2+G(\phi_1)=0, \;\;\,x\in\mathbb{R},\\
&\phi_1(-\yy)=\phi_2(-\yy)=0, ~ \phi_1(\yy)=u^*, ~ \phi_2(\yy)=v^*,
\end{aligned}\right.
 \eess
 where $(u^*,v^*)$ is the unique positive root of $au=H(v)$ and $bv=G(u)$. Moreover, the dynamics of the corresponding ODE system with positive initial value is governed by $\mathcal{R}_0$. Namely, when $\mathcal{R}_0<1$, $(0,0)$ is globally asymptotically stable; while when $\mathcal{R}_0>1$, then the unique positive equilibrium $(u^*,v^*)$ is globally asymptotically stable.

 If $H(v)=cv$, then system \eqref{1.1} reduces to
 \bess
\left\{\!\begin{aligned}
&u_{t}=d_1 u_{xx}-au+cv, & &t>0,~x\in\mathbb{R},\\
&v_{t}=d_2 v_{xx}-bv+G(u), & &t>0,~x\in\mathbb{R},
\end{aligned}\right.
 \eess
where $G$ satisfies that $G\in C^2([0,\yy))$, $G(0)=0<G'(u)$ in $[0,\yy)$, $G(u)/u$ is strictly decreasing in $(0,\yy)$ and $\lim_{u\to\yy}{G(u)}/{u}<{ab}/{c}$. The corresponding ODE system was first proposed in \cite{CP} to describe the 1973 cholera epidemic spread in the European Mediterranean regions.

When we model the spread of epidemics mathematically, one of important issues is to know where the spreading front of an epidemic is located, which naturally motivates us to discuss the related systems on the domain whose boundary is unknown and varies over time, instead of the fixed boundary domain or the whole space. Inspired by the pioneering work \cite{DL}, where the Stefan boundary condition is successfully introduced to a reaction-diffusion equation arising from biology and many new as well as interesting propagation phenomena were found, Li et al \cite{LL} recently incorporated the Setan boundary condition into \eqref{1.1} and proposed the following model
  \bes\begin{cases}\label{1.2}
u_t=d_1u_{xx}-au+H(v), &t>0,~x\in(0,h(t)),\\
v_t=d_2v_{xx}-bv+G(u), &t>0,~x\in(0,h(t)),\\
\mathbb{B}[u](t,0)=\mathbb{B}[v](t,0)=u(t,h(t))=v(t,h(t))=0, \; &t>0,\\
h'(t)=-\mu_1 u_x(t,h(t))-\mu_2v_x(t,h(t)), & t>0,\\
h(0)=h_0, ~ u(0,x)=u_{0}(x), ~ v(0,x)=v_0(x),&0\le x\le h_0,
 \end{cases}
 \ees
 where the initial functions $u_0$ and $v_0$ satisfy
  \begin{enumerate}
\item[{\bf(I)}]\; $w\in C^{2}([0,h_0])$, $w'(0)>0$, $w(h_0)=w(0)=0<w(x)$ in $(0,h_0)$ when $\mathbb B[w]=w$, $w(h_0)=w'(0)=0<w(x)$ in $[0,h_0)$ when $\mathbb{B}[w]=w_x$.
 \end{enumerate}
 The operator $\mathbb{B}[w]=w$ or $w_x$, which indicates that the homogeneous Dirichlet or Neumann boundary condition are imposed at the fixed boundary $x=0$, respectively.
They showed that longtime behaviors are governed by a spreading-vanishing dichotomy. That is, one of the following alternatives must happen:

{\it $\underline {Spreading}$:} necessarily $\mathcal{R}_0>1$, $\lim_{t\to\yy}h(t)=\yy$,
\bess\left\{\!\begin{array}{ll}
\dd\lim_{t\to\yy}(u(t,x),v(t,x))=(U(x),V(x)) ~ {\rm in ~ }C_{\rm loc}([0,\yy)) {\rm ~ when ~ operator ~ }\mathbb{B}[w]=w,\\
\dd\lim_{t\to\yy}(u(t,x),v(t,x))=(u^*,v^*) ~ {\rm in ~ }C_{\rm loc}([0,\yy)) {\rm ~ when ~ operator ~ }\mathbb{B}[w]=w_x,
 \end{array}\right.
\eess
where $(U(x),V(x))$ is the unique bounded positive solution of
 \bess\left\{\!\begin{array}{ll}
-d_1u''=-au+H(v), \;\;&x\in(0,\yy),\\[1mm]
-d_2v''=-bv+G(u), &x\in(0,\yy),\\[1mm]
u(0)=v(0)=0.
 \end{array}\right.
\eess

{\it $\underline {Vanishing}$:} $\lim_{t\to\yy}h(t)<\yy$, $\lim_{t\to\yy}\|u(t,x)+v(t,x)\|_{C([0,h(t)])}=0$.

Moreover, some sharp criteria for spreading and vanishing were also obtained. For example, spreading will happen if $h_0\ge l_0$, where $l_0$ is given by
\bes\label{1.3}\left\{\!\begin{array}{ll}
l_0=\dd\pi\sqrt{\frac{ad_2+bd_1+\sqrt{(ad_2-bd_1)^2
 +4d_1d_2H'(0)G'(0)}}{2(H'(0)G'(0)-ab)}}\;\;{\rm ~ when ~ operator ~ }\mathbb{B}[w]=w,\\[4mm]
 l_0=\dd\frac{\pi}{2}\sqrt{\frac{ad_2+bd_1+\sqrt{(ad_2-bd_1)^2
 +4d_1d_2H'(0)G'(0)}}{2(H'(0)G'(0)-ab)}}\;\; {\rm ~ when ~ operator ~  }\mathbb{B}[w]=w_x.
 \end{array}\right.
\ees

Later, when spreading happens, we in \cite{LLW} derived the spreading speed of \eqref{1.2} by analyzing the following semi-wave problem
 \bes\label{1.4}
\left\{\!\begin{aligned}
&d_1\varphi''-c\varphi'-a\varphi+H(\psi)=0, \;\;x>0,\\
&d_2\psi''-c\psi'-b\psi+G(\varphi)=0, \;\;\,x>0,\\
&\varphi(0)=\psi(0)=0, ~ \varphi(\yy)=u^*, ~ \psi(\yy)=v^*.
\end{aligned}\right.
 \ees
More precisely, we first proved that for any $c\in[0,c_*)$, \eqref{1.4} has a unique solution $(\varphi,\psi)$ with $\varphi'(x)>0$ and $\psi'(x)>0$ for all $x\ge0$; moreover, for any positive constants $\mu_1$ and $\mu_2$, there exists  a unique $c_0\in(0,c_*)$ such that
\[\mu_1\varphi'_{c_0}(0)+\mu_2\psi'_{c_0}(0)=c_0,\]
where $(\varphi_{c_0},\psi_{c_0})$ is the unique monotone solution of \eqref{1.4} with $c=c_0$.
Then with the aid of a complete understanding for \eqref{1.4}, it was shown that if spreading occurs, the unique solution $(u,v,h)$ of \eqref{1.2} satisfies
\bes\label{1.5}\begin{cases}
\dd\lim_{t\to\yy}\frac{h(t)}{t}=c_{0},\\
\dd\lim_{t\to \yy}\max_{x\in[0,ct]}\big(|u(t,x)-U(x)|+|v(t,x)-V(x)|\big)=0, ~ \forall\, c\in[0,c_0) ~ {\rm when ~} \mathbb{B}[w]=w,\\
\dd\lim_{t\to \yy}\max_{x\in[0,ct]}\big(|u(t,x)-u^*|+|v(t,x)-v^*|\big)=0, ~ \forall\, c\in[0,c_0) ~ {\rm when ~} \mathbb{B}[w]=w_x.
 \end{cases}
\ees

In this paper, motivated by the results of \cite{DMZ,KMY1,WND1}, we aim at obtaining the sharp asymptotic profile of the solution $(u,v,h)$ when spreading happens. The main result is the following theorem.

\begin{theo}\label{t1.1}Let spreading happen and $(u,v,h)$ be the unique solution of \eqref{1.2}.  Then the following statements are valid.
\begin{enumerate}[$(1)$]
\item If $\mathbb{B}[w]=w$ and $(d_1,a,H'(v^*))=(d_2,b,G'(u^*))$, then there exists some $h_*\in\mathbb{R}$ such that
\bess
\begin{cases}h(t)=c_0t+h_*+o(1),\\
\dd\lim_{t\to\yy}\max_{x\in[ct,h(t)]}\big(|u(t,x)-\varphi_{c_0}(h(t)-x)|+|v(t,x)-\psi_{c_0}(h(t)-x)|\big)=0, ~ ~ \forall c\in(0,c_0).
\end{cases}
\eess
\item If $\mathbb{B}[w]=w_x$, then there exists some $\tilde h_*\in\mathbb{R}$ such that
\bess
\begin{cases}h(t)=c_0t+\tilde h_*+o(1),\\
\dd\lim_{t\to\yy}\max_{x\in[0,h(t)]}\big(|u(t,x)-\varphi_{c_0}(h(t)-x)|+|v(t,x)-\psi_{c_0}(h(t)-x)|\big)=0.
\end{cases}
\eess
\end{enumerate}
\end{theo}

Such sharp convergence results for reaction-diffusion equation (system) with or without free boundary have been attracting much attention of lots of scholars over the past decades. For example, please see \cite{HNRR} for the Fisher-KPP equation, \cite{PWZ,WXZ,AX} for the Lotka-Volterra competition system, \cite{KMY,LiuLou} for the bistable equation with a free boundary (especially, the so-called logarithmic shifting appears in \cite{KMY}), and \cite{WQW} for the Lotka-Volterra competition system with free boundary under strong-weak competition case, \cite{WND1} for the West Nile virus model with double free boundaries. This paper seems to be the first attempt to consider the sharp asymptotic spreading profile of a epidemic model with one free boundary and one fixed boundary that is subject to the homogeneous Dirichlet boundary condition and Neumann boundary condition, respectively.

Before starting our proof of Theorem \ref{t1.1}, we would like to mention that incorporating free boundary condition into the reaction-diffusion equation (system) has been increasingly recognized and much progress has been made over the past decades. Particularly, Cao et al \cite{CDLL} assumed that the species in the model of \cite{DL} adopt a nonlocal diffusion strategy, and thus proposed a new free boundary problem. An analogous problem was meanwhile proposed by Cor{\'t}azar et al \cite{CQW}. For the free boundary problem with nonlocal diffusion, it was proved in \cite{DLZ,DN1,DN2} that accelerated spreading may happen, i.e., the spreading speed may be infinite, depending on whether a threshold condition is satisfied by kernel function, which is much different from those of the free boundary problems with local diffusion since they always have a finite spreading speed (see \cite{DL,DLou,DWZ17,WND}). For more details, one can refer to \cite{DuG,GW15,LWZjde22,DFS} and the references therein, or the expository article \cite{Du22}.

\begin{remark}\label{r1.1}We mention that when spreading happens and $\mathbb{B}[w]=w$, in contrast to the models with double free boundaries in \cite{WND1}, where the solution will converge to a unique constant equilibrium, the solution component $(u,v)$ of \eqref{1.2} will converge to a unique non-constant equilibrium $(U(x),V(x))$. This difference brings a considerable challenge for the constructing of a desired lower solution. To overcome this difficulty, we will prove a crucial estimate for solution component $(u,v)$, which is non-trivial and needs a extra technical assumption $(d_1,a,H'(v^*))=(d_2,b,G'(u^*))$.
\end{remark}

\begin{remark}\label{r1.2} Form Theorem \ref{t1.1}, it follows that whether $\mathbb{B}[w]=w$ or $\mathbb{B}[w]=w_x$, the asymptotic behaviors of $h(t)$ and $(u,v)$ near the spreading front $h(t)$ are the same.
\end{remark}

This paper is arranged as follows. Section 2 involves some preliminary work, including some properties of solution $(u,v,h)$ of \eqref{1.2}, an estimate for the convergence rate for the solution $(\varphi_{c_0},\psi_{c_0})$ of semi wave problem \eqref{1.4} and a crucial lower bound for solution component $(u,v)$ when $\mathbb{B}[w]=w$. Section 3 is devoted to the proof of boundedness of $h(t)-c_0t$ by utilizing the estimates in Section 2 to constructing some suitable upper and lower solutions. In Section 4, we complete the proof of Theorem \ref{t1.1} by following similar lines as in \cite{WND1}. Throughout this paper, we always assume that spreading happens for \eqref{1.2}, which implies $\mathcal{R}_0>1$.

\section{Preliminaries}{\setlength\arraycolsep{2pt}
In this section, we give some useful lemmas which will pave the road for later discussion. Note that comparison arguments will often be used throughout this paper. However, we do not give the related comparison principles since they are well-known, and the readers can refer to \cite[Lemmas 2.1 and 2.2]{WND1} for some examples. There are some notations that need to be introduced. For any vectors ${\bf u}=(u_1,u_2,\cdots, u_m)$ and ${\bf v}=(v_1,v_2,\cdots,v_m)$, ${\bf u}\ge{\bf v}$ (${\bf u}>{\bf v}$) means $u_i\ge v_i$ ($u_i>v_i$) for $i=1,2,\cdots, m$; $Q^T$ represents the transpose of matrix $Q$.

The first lemma concerns the properties of the unique solution $(u,v,h)$ of \eqref{1.2}. Since the proof is standard, we omit the details and the interested readers can refer to \cite[Lemma 2.2]{DL} for the boundedness of $(u,v)$ and $h'$, and \cite[Theorem 8.6, pp 198]{Wpara} for the regularity of $(u,v,h)$.

\begin{lem}\label{l2.0}Let $(u,v,h)$ be the unique solution of \eqref{1.2}. Then there exists a $C>0$ depending only on the parameters of \eqref{1.2}, nonlinear term $(H,G)$ and the initial functions $u_0$ and $v_0$, such that
\bess
0<u(t,x),v(t,x)\le C ~ ~ {\rm and ~ ~ }0<h'(t)<C
\eess
for $t>0$ and $x\in(0,h(t))$. Moreover, for any $\alpha\in(0,1)$, and $\tau,T>0$, we have
\bess
h\in C^{1+\frac{1+\alpha}{2}}([\tau,T]) {\rm ~ ~ and ~ ~}(u,v)\in[C^{1+\frac{\alpha}{2},~ 2+\alpha}([\tau,T]\times[0,h(t)])]^2.
\eess
\end{lem}

Then we show the rate of convergence of solution $(\varphi,\psi)$ of semi-wave problem \eqref{1.4}.

\begin{lem}\label{l2.1} Let $(\varphi,\psi)$ be the unique monotone solution of semi-wave problem \eqref{1.4}. Then there exists a positive constant $\alpha$, depending only on the parameters in \eqref{1.4} as well as the nonlinear terms $H$ and $G$, such that
\bes\label{2.1}
u^*-\varphi(x)+v^*-\psi(x)+\varphi'(x)+\psi'(x)=O(e^{-\alpha x}).
\ees
\end{lem}
\begin{proof}Note that $0\le\varphi(x)<u^*$, $0\le \psi(x)<v^*$, $\varphi'(x)>0$ and $\psi'(x)>0$ for $x\ge0$. By \cite[Theorem 1.1]{LLW}, we have $(u^*-\varphi(x),v^*-\psi(x))=e^{-\beta x}(p+o(1),q+o(1))$ for some positive constants $\beta$, $p$ and $q$, which implies that $u^*-\varphi(x)+v^*-\psi(x)=O(e^{-\beta x})$. Hence to prove \eqref{2.1}, it suffices to show $\varphi'(x)+\psi'(x)=O(e^{-\beta x})$.

Let us prove $\varphi'(x)=O(e^{-\beta x})$. Note that $au^*=H(v^*)$. It is easy to see that $a\varphi(x)-H(\psi(x))=O(e^{-\beta x})$. For clarity, denote $f(x)=\frac{-a\varphi(x)+H(\psi(x))}{d_1}$. In view of the identity of $\varphi$, we have
\[\varphi''-\frac{c}{d_1}\varphi'+f(x)=0, ~ ~ x>0.\]
Choose $\ep$ to be small enough such that $\ep\in(0,{c}/{d_1})$. Since $\varphi'(x)>0$ for $x\ge0$, we have
\bess
\varphi''-\ep\varphi'+f(x)>0, ~ ~ x>0.
\eess
Multiplying $e^{-\ep x}$ to the above identity and integrating it from $x$ to $A$ yield
\bess
e^{-\ep A}\varphi'(A)-e^{-\ep x}\varphi'(x)+\int_{x}^{A}e^{-\ep y}f(y)\dy>0.
\eess
Multiplying $e^{(\beta+\ep) x}$ to the above inequality gives
\bess
e^{\beta x}\varphi'(x)&\le& e^{(\beta+\ep) x-\ep A}\varphi'(A)+e^{(\beta+\ep) x}\int_{x}^{A}e^{-\ep y}f(y)\dy\\
&\le& e^{(\beta+\ep) x-\ep A}\varphi'(A)+Ce^{(\beta+\ep) x}\int_{x}^{A}e^{-(\ep+\beta) y}\dy\\
&=& e^{(\beta+\ep) x-\ep A}\varphi'(A)+\frac{C}{\ep+\beta}(1-e^{(\beta+\ep) x-(\ep+\beta)A}).
\eess
Notice that $\varphi'(x)$ is bounded for $x\ge0$. Letting $A\to\yy$, we obtain
\bess
e^{\beta x}\varphi'(x)\le\frac{C}{\ep+\beta} ~ ~ {\rm for ~ }x>0,
\eess
which implies that $\varphi'(x)=O(e^{-\beta x})$. Similarly, it is not hard to show that $\psi'(x)=O(e^{-\beta x})$. Thus \eqref{2.1} is obtained and the proof is finished.
\end{proof}

Next we give a lower bound for solution component $(u,v)$ when $\mathbb{B}[w]=w$ that is crucial for the constructing of lower solution. The idea of proof comes from \cite[Lemma 2.8]{PWZ} or \cite[Lemma 6.5]{DLou}.

\begin{lem}\label{l2.2}Let $(u,v,h)$ be the unique solution of \eqref{1.2} with $\mathbb{B}[w]=w$. Assume $(d_1,a,H'(v^*))=(d_2,b,G'(u^*))$. Then for any $0<c_1<c_2<c_0$, there exist positive constants $\delta$, $M$ and $T$ such that
\bess
u(t,x)\ge u^*-Me^{-\delta t} ~ {\rm and ~} v(t,x)\ge v^*-Me^{-\delta t} ~ ~ {\rm for ~ }t\ge T, ~ x\in[c_1t,c_2t].
\eess
\end{lem}

\begin{proof}Since spreading happens, we can find a $T_1>0$ such that $h(T_1)\ge l_0$, where $l_0$ is defined by \eqref{1.3}. Then we consider the following problem
 \bes\begin{cases}\label{2.2}
\tilde u_t=d_1\tilde u_{xx}-a\tilde u+H(\tilde v), &t>0,~x\in(0,\tilde h(t)),\\
\tilde v_t=d_1\tilde v_{xx}-b\tilde v+G(\tilde u), &t>0,~x\in(0,\tilde h(t)),\\
\tilde u(t,0)=\tilde v(t,0)=\tilde u(t,\tilde h(t))=\tilde v(t,\tilde h(t))=0, \; &t>0,\\
\tilde h'(t)=-\mu_1 \tilde u_x(t,\tilde h(t))-\mu_2\tilde v_x(t,\tilde h(t)), & t>0,\\
\tilde h(0)=h(T_1), ~ \tilde u(0,x)=\tilde u_{0}(x), ~ \tilde v(0,x)=\tilde v_0(x),&0\le x\le h(T_1),
 \end{cases}
 \ees
where the initial functions $\tilde{u}_0$ and $\tilde{v}_0$ satisfy condition {\bf(I)} and
\[\tilde{u}_0(x)\le \min\{u(T_1,x),u^*\}, ~ ~ \tilde{v}_0(x)\le\min\{v(T_1,x),v^*\} ~ ~ {\rm for ~ }x\in[0,h(T_1)].\]
By a comparison argument, we have
\bes\label{2.3}
(\min\{u^*,u(t+T_1,x)\},\min\{v^*, v(t+T_1,x)\})\ge(\tilde{u}(t,x),\tilde{v}(t,x)) ~ ~ {\rm for ~ }t\ge0, ~ x\in[0,\tilde{h}(t)].\ees
Moreover, due to $\tilde{h}(0)=h(T_1)>l_0$, spreading happens for \eqref{2.2} and \eqref{1.5} holds for $(\tilde{u},\tilde{v},\tilde{h})$.

Next we estimate $(\tilde{u},\tilde{v})$. For clarity, denote
\[f_1(\tilde u,\tilde v)=-a\tilde u+H(\tilde v), ~ ~ f_2(\tilde u,\tilde v)=-b\tilde v+G(\tilde u).\]
Thanks to {\bf (H)}, $a=b$ and $H'(v^*)=G'(u^*)$, it is easy to see that $a=b>H'(v^*)=G'(u^*)$. Thus we can choose positive constants $\gamma_1$ and $\gamma_2$ satisfying  $H'(v^*)<\gamma_1<\gamma_2<a$. Simple calculations show that there exists a small $\ep>0$ such that when $u^*-\tilde{u}\le\ep$ and $v^*-\tilde{v}\le\ep$,
\bess
f_1(\tilde u,\tilde v)&=&f_1(\tilde u,\tilde v)-f_1(u^*,v^*)=(\partial_{\tilde{u}} f_1(u^*,v^*)+o(1))(\tilde u-u^*)+(\partial_{\tilde{v}} f_1(u^*,v^*)+o(1))(\tilde v-v^*)\\
&=&(-a+o(1))(\tilde u-u^*)+(H'(v^*)+o(1))(\tilde v-v^*)\\
&\ge&-\gamma_2(\tilde u-u^*)+\gamma_1(\tilde v-v^*)
\eess
where $o(1)\to0$ as $(\tilde{u},\tilde{v})\to(u^*,v^*)$. Similarly, we can obtain that for some small $\ep>0$, there holds
\[f_2(\tilde u,\tilde v)\ge-\gamma_2(\tilde v-v^*)+\gamma_1(\tilde u-u^*) ~ ~ {\rm for ~ }\tilde u\in[u^*-\ep,u^*], ~ \tilde
v\in[v^*-\ep,v^*].\]

 Notice that $(U(x),V(x))\to (u^*,v^*)$ as $x\to\yy$. By \eqref{1.5}, there exists a $T_2>T_1$ such that
\[h(t)>c_2t, ~ ~(\tilde u(t,x),\tilde v(t,x))\ge (u^*-\ep/2,v^*-\ep/2) ~ ~ {\rm for ~ }t\ge T_2, ~ x\in[c_1t,c_2t].\]
Fix $T\ge T_2$ and denote $\Psi=\tilde u+\tilde v$. Then $\Psi$ satisfies
 \bess\begin{cases}
 \Psi_t\ge d_1\Psi_{xx}-(\gamma_2-\gamma_1)\Psi+(\gamma_2-\gamma_1)(u^*+v^*), &t\ge T, ~ x\in(c_1T,c_2T),\\
\Psi(t,c_1T)\ge u^*+v^*-\ep, ~ \Psi(t,c_2T)\ge u^*+v^*-\ep, &t\ge T,\\
 \Psi(T,x)\ge u^*+v^*-\ep, &x\in[c_1T,c_2T].
 \end{cases}
 \eess
 Consider the following auxiliary problem
 \bess\begin{cases}
 \phi_t=d_1\phi_{xx}-(\gamma_2-\gamma_1)\phi+(\gamma_2-\gamma_1)(u^*+v^*), &t>0, ~ x\in(c_1T,c_2T),\\
 \phi(t,c_1T)= u^*+v^*-\ep, ~ \phi(t,c_2T)=u^*+v^*-\ep, &t>0,\\
 \phi(0,x)=u^*+v^*-\ep, &x\in[c_1T,c_2T].
 \end{cases}
 \eess
 In light of a comparison argument, we have $\Psi(t+T,x)\ge\phi(t,x)$ for $t\ge0$ and $x\in[c_1T,c_2T]$.

 Let
 \[\Phi(t,x)=e^{(\gamma_2-\gamma_1)t}[\phi(t,x+\frac{c_1T+c_2T}{2})-(u^*+v^*-\ep)], ~ ~ \bar{c}=\frac{c_2-c_1}{2}.\]
 Then clearly $\Phi$ satisfies
  \bess\begin{cases}
 \Phi_t=d_1\Phi_{xx}+(\gamma_2-\gamma_1)\ep e^{(\gamma_2-\gamma_1)t}, &t>0, ~ x\in(-\bar{c}T,\bar{c}T),\\
 \Phi(t,\pm\bar{c}T)=0, &t>0,\\
 \Phi(0,x)=0, &x\in[-\bar c T,\bar c T].
 \end{cases}
 \eess
 By the Green function of the heat equation, we have
 \[\Phi(t,x)=(\gamma_2-\gamma_1)\ep\int_{0}^{t}e^{(\gamma_2-\gamma_1)\tau}\int_{-\bar{c}T}^{\bar{c}T}\tilde{G}(t,x;\tau,\xi){\rm d}\xi{\rm d}\tau, ~t>0, ~ -\bar{c}T<x<\bar{c}T,\]
 where $\tilde{G}(t,x;\tau,\xi)$ is the Green function (see, e.g. \cite{Fri}, pp 84) defined by
 \[\tilde{G}(t,x;\tau,\xi)=\sum_{n\in\mathbb{Z}}(-1)^nG(t-\tau,x-\xi-2n\bar{c}T)\]
 with the heat kernel $G$ given by
 \[G(t,x)=\frac{1}{\sqrt{4\pi d_1t}}e^{-\frac{x^2}{4d_1t}}.\]
 In view of \cite[Lemma 2.8]{PWZ}, for any $\ep_1\in(0,1)$, there exists a $T_3\gg\max\{T_2,1\}$ such that for all $T\ge T_3$,
 \[\int_{-\bar{c}T}^{\bar{c}T}\tilde{G}(t,x;\tau,\xi){\rm d}\xi\ge1-\frac{4}{\sqrt{\pi}}e^{-\frac{T}{2\sqrt{d_1}}}\]
 holds for all $(t,x)\in D_{\ep_1}$ defined by
 \[D_{\ep_1}=\{(t,x):0<t\le \frac{\ep^2_1\bar{c}^2T}{4\sqrt{d_1}},~ |x|\le(1-\ep_1)\bar{c}T\}.\]
 Hence we get
 \bess
 \Phi(t,x)\ge \ep(e^{(\gamma_2-\gamma_1)t}-1)(1-\frac{4}{\sqrt{\pi}}e^{-\frac{T}{2\sqrt{d_1}}})
 \eess
 for $(t,x)\in D_{\ep_1}$. Furthermore, we have that for all $(t,x)\in D_{\ep_1}$,
 \bess
 \phi(t,x+\frac{c_1T+c_2T}{2})&=&e^{-(\gamma_2-\gamma_1)t}\Phi(t,x)+u^*+v^*-\ep\\
 &\ge&\ep(1-e^{-(\gamma_2-\gamma_1)t})(1-\frac{4}{\sqrt{\pi}}e^{-\frac{T}{2\sqrt{d_1}}})+u^*+v^*-\ep\\
 &\ge& u^*+v^*-\ep e^{-(\gamma_2-\gamma_1)t}-\ep\frac{4}{\sqrt{\pi}}e^{-\frac{T}{2\sqrt{d_1}}}
 \eess
 which implies that for all $(t,x)\in \tilde{D}_{\ep_1}$ defined by
\[\tilde{D}_{\ep_1}=\{(t,x): 0<t\le \frac{\ep^2_1\bar{c}^2T}{4\sqrt{d_1}}, ~ (c_1+\ep_1\bar{c})T\le x\le (c_2-\ep_1\bar{c})T\},\]
 there holds:
 \[ \phi(t,x)\ge u^*+v^*-\ep e^{-(\gamma_2-\gamma_1)t}-\ep\frac{4}{\sqrt{\pi}}e^{-\frac{T}{2\sqrt{d_1}}}.\]
 Choose $t=\frac{\ep^2_1\bar{c}^2T}{4\sqrt{d_1}}$ and take $\ep_1$ small enough such that
 \[\frac{(\gamma_2-\gamma_1)\ep^2_1\bar{c}^2}{4\sqrt{d_1}}\le\frac{1}{2\sqrt{d_1}}.\]
 Then  for $ (c_1+\ep_1\bar{c})T\le x\le (c_2-\ep_1\bar{c})T$,
\bess
 \phi(\frac{\ep^2_1\bar{c}^2T}{4\sqrt{d_1}},x)&\ge& u^*v^*-\ep e^{-(\gamma_2-\gamma_1)\frac{\ep^2_1\bar{c}^2T}{4\sqrt{d_1}}}-\ep{\sqrt{\pi}}e^{-\frac{T}{2\sqrt{d_1}}}\\
 &\ge&u^*+v^*-\ep(1+\frac{4}{\sqrt{\pi}})e^{-(\gamma_2-\gamma_1)\frac{\ep^2_1\bar{c}^2T}{4\sqrt{d_1}}}.
\eess
Thus
\bess
\Psi(\frac{\ep^2_1\bar{c}^2T}{4\sqrt{d_1}}+T,x)\ge u^*+v^*-\ep(1+\frac{4}{\sqrt{\pi}})e^{-(\gamma_2-\gamma_1)\frac{\ep^2_1\bar{c}^2T}{4\sqrt{d_1}}}
\eess
with $ (c_1+\ep_1\bar{c})T\le x\le (c_2-\ep_1\bar{c})T$.
Let
\[t=\frac{\ep^2_1\bar{c}^2T}{4\sqrt{d_1}}+T.\]
Then we have
\bess
\Psi(t,x)\ge u^*+v^*-Me^{-\delta t}
\eess
for
\[ t\ge (1+\frac{\ep^2_1\bar{c}^2}{4\sqrt{d_1}}) T_3, ~ (c_1+\ep_1\bar{c})(1+\frac{\ep^2_1\bar{c}^2}{4\sqrt{d_1}})^{-1}t\le x\le (c_2-\ep_1\bar{c})(1+\frac{\ep^2_1\bar{c}^2}{4\sqrt{d_1}})^{-1}t,\]
where
\[M=\ep(1+\frac{4}{\sqrt{\pi}}), ~ ~ \delta=(\gamma_2-\gamma_1)\frac{\ep^2_1\bar{c}^2}{4\sqrt{d_1}}(1+\frac{\ep^2_1\bar{c}^2}{4\sqrt{d_1}})^{-1}.\]
Combing with the arbitrariness of $c_1$, $c_2$ and $\ep_1$, we have proved that for any given $0<c_1<c_2<c_0$, there exist $M$, $\delta$ and $T$ such that
\[\tilde{u}(t,x)+\tilde{v}(t,x)\ge u^*+v^*-M e^{-\delta t} ~ ~ {\rm for ~ }t\ge T, ~ x\in[c_1t,c_2t],\]
which together with \eqref{2.3} leads to
\[(u(t,x),v(t,x))\ge (u^*-Me^{-\delta (t-T_1)}, v^*-Me^{-\delta (t-T_1)}) ~ {\rm for ~ }t\ge T+T_1, ~ x\in[c_1(t-T_1),c_2(t-T_1)].\]
Note that $c_1,c_2\in(0,c_0)$ are arbitrary and $T_1$ depends only on $l_0$. We thus complete the proof.
\end{proof}

\begin{remark}\label{r2.1}It is easy to show that there exist positive constants $M$ and $\delta$ such that the solution component $(u,v)$ of \eqref{1.2} satisfies $(u(t,x),v(t,x))\le (u^*+Me^{-\delta t},v^*+Me^{-\delta t})$ for $t\ge0$ and $x\ge0$. In fact, we can verify that $(u^*+M_1e^{-\delta t},v^*+M_2e^{-\delta t})$ is an upper solution for \eqref{1.2}, where $M_1$, $M_2$ and $\delta$ satisfy
\bess
&u^*+M_1\ge\|u_0\|_{L^{\yy}([0,h_0])}, ~ v^*+M_2\ge\|v_0\|_{L^{\yy}([0,h_0])}, ~ \frac{b}{G'(u^*)}>\frac{M_1}{M_2}>\frac{H'(v^*)}{a}\\
&\delta\le \min\{\frac{aM_1-H'(v^*)M_2}{M_1}, \frac{bM_2-G'(u^*)M_1}{M_2}\}.
\eess
Then we take $M=\max\{M_1,M_2\}$.
\end{remark}

At the end of this section, we introduce a technical lemma about the property of the nonlinear terms $H$ and $G$.

\begin{lem}[{\cite[Lemma 2.9]{DN3}}]\label{l2.3}Assume that $F=(f_i)\in C^2(\mathbb{R}^m,\mathbb{R}^m)$, ${\bf u^*}>0$, $F({\bf u^*})={\bf 0}$ and ${\bf u^*}[\nabla F({\bf u^*})]^T<{\bf 0}$.
Then there exists a $\delta_0>0$, depending only on $F$, such that for any $0<\ep\ll1$ and ${\bf u},~ {\bf v}\in[(1-\delta_0){\bf u^*},{\bf u^*}]$ satisfying
\[(u^*_i-u_i)(u^*_j-v_j)\le C\delta_0\ep ~ {\rm for ~ some ~ }C>0 {\rm ~ and ~ all ~ }i,~ j\in\{1,2,\cdots,m\},\]
we have
\[(1-\ep)\left[F({\bf u})+F({\bf v})\right]-F((1-\ep)({\bf u}+{\bf v}-{\bf u^*}))\le\frac{\ep}{2}{\bf u^*}[\nabla F({\bf u^*})]^T.\]
\end{lem}

\section{The bound for $h(t)-c_0t$}
In this section, by constructing some suitably upper and lower solutions, we will prove that $h(t)-c_0t$ is bounded for all $t\ge0$. Namely, we are going to prove the following result.

\begin{prop}\label{p3.1} Let $(u,v,h)$ be the unique solution of \eqref{1.2}. If $\mathbb{B}[w]=w$ and $(d_1,a,H'(v^*))=(d_2,b,G'(u^*))$, or $\mathbb{B}[w]=w_x$, then there exists a positive constant $C$, depending only on the parameters of \eqref{1.2}, nonlinear terms $H$ and $G$ as well as the initial functions $u_0$ and $v_0$, such that
\[|h(t)-c_0t|\le C ~ ~{\rm for ~ all ~ }t\ge0.\]
\end{prop}

This proposition will be proved by the following three lemmas. We first give an upper bound.

\begin{lem}\label{l3.1}Let $(u,v,h)$ be the unique solution of \eqref{1.2}. Then there is a $C>0$ such that
\[h(t)-c_0t\le C ~ ~ {\rm for ~ }t\ge0.\]
\end{lem}
\begin{proof}We will prove this assertion by constructing a suitably upper solution. For any given $T>0$ and $X>0$, we define
\bess
&\bar{h}(t)=c_0(t-T)+\sigma(1-e^{-\delta(t-T)})+h(T)+X,\\
&\bar{u}(t,x)=(1+Ke^{-\delta(t-T)})\varphi_{c_0}(\bar{h}(t)-x), ~ ~ \bar{v}(t,x)=(1+Ke^{-\delta(t-T)})\psi_{c_0}(\bar{h}(t)-x),
\eess
where positive constants $\sigma$, $\delta$ and $K\ge1$ are determined later.

Next we show that there are suitable $\sigma$, $\delta$ and $K$ such that for any $T>0$ and $X>0$, $(\bar{u},\bar{h})$ satisfies
\bes\label{3.1}\begin{cases}
\bar u_t\ge d_1\bar u_{xx}-a\bar u+H(\bar v), &t>T,~x\in(0,\bar h(t)),\\
 \bar v_t\ge d_2\bar v_{xx}-b\bar v+G(\bar u), &t>T,~x\in(0,\bar h(t)),\\
\bar u(t,\bar h(t))\ge0, ~ \bar v(t,\bar h(t))\ge0, \; &t>T,\\
\bar u(t,0)\ge 0, ~ \bar v(t,0)\ge 0, \; &t>T, ~ {\rm if ~ }\mathbb{B}[w]=w,\\
\bar u_x(t,0)\le 0, ~ \bar v_x(t,0)\le 0, \; &t>T, ~ {\rm if ~ }\mathbb{B}[w]=w_x,\\
\bar h'(t)\ge-\mu_1 \bar u_x(t,\bar h(t))-\mu_2\bar v_x(t,\bar h(t)), & t>T,\\
\bar h(T)\ge h(T), ~ \bar u(T,x)\ge u(T,x), ~ \bar v(T,x)\ge v(T,x),&0\le x\le h(T).
 \end{cases}
 \ees
 Once we have proved \eqref{3.1}, using a comparison argument yields $\bar{h}(t)\ge h(t)$ and $(\bar{u}(t,x),\bar{v}(t,x))\ge(u(t,x),v(t,x))$ for $t\ge T$ and $x\in[0,h(t)]$, which clearly implies our desired result.

 Now we prove \eqref{3.1}. By the definition of $(\bar{u},\bar{h})$, we easily see that $\bar{u}(t,\bar h(t))=\bar{v}(t,\bar{h}(t))=0$, $\bar u(t,0)\ge0$, $\bar{u}(t,0)\ge0$, $\bar{u}_x(t,0)<0$, $\bar{v}_x(t,0)<0$ and $\bar{h}(T)>h(T)$. Direct computations show
 \bess
-\mu_1 \bar u_x(t,\bar h(t))-\mu_2\bar v_x(t,\bar h(t))&=&(1+Ke^{-\delta(t-T)})(\mu_1\varphi_{c_0}'(0)+\mu_2\psi_{c_0}'(0))\\
&=&c_0(1+Ke^{-\delta(t-T)})\le c_0+\sigma\delta e^{-\delta(t-T)}=\bar{h}'(t)
 \eess
 provided that $c_0K\le \sigma\delta$. Moreover, for any $X>0$, we can let $K$ be large enough such that for $x\in[0,h(T)]$, we have
 \bes\label{3.2}
 \bar{u}(T,x)\ge(1+K)\varphi_{c_0}(X)\ge C\ge u(T,x) ~ {\rm and ~ }\bar{v}(T,x)\ge(1+K)\psi_{c_0}(X)\ge C\ge v(T,x).
 \ees
 Hence it remains to show that the first two inequalities are valid. For convenience, denote $s=\bar{h}(t)-x$. Straightforward calculations yield
 \bess
 &&\bar{u}_t-d_1\bar{u}_{xx}+a\bar{u}-H(\bar{v})\\
 &&=-\delta Ke^{-\delta(t-T)}\varphi_{c_0}(s)+(c_0+\sigma \delta e^{-\delta(t-T)})(1+Ke^{-\delta(t-T)})\varphi_{c_0}'(s)-d_1(1+Ke^{-\delta(t-T)})\varphi_{c_0}''(s)\\
 && \;\;\;+a(1+Ke^{-\delta(t-T)})\varphi_{c_0}(s)-H((1+Ke^{-\delta(t-T)})\psi_{c_0}(s))\\
 &&=-\delta Ke^{-\delta(t-T)}\varphi_{c_0}(s)+\sigma\delta e^{-\delta(t-T)}(1+Ke^{-\delta(t-T)})\varphi_{c_0}'(s)\\
 &&\;\;\;+(1+Ke^{-\delta(t-T)})[c_0\varphi_{c_0}'(s)-d_1\varphi_{c_0}''(s)+a\varphi_{c_0}(s)]-H((1+Ke^{-\delta(t-T)})\psi_{c_0}(s))\\
 &&=-\delta Ke^{-\delta(t-T)}\varphi_{c_0}(s)+\sigma\delta e^{-\delta(t-T)}(1+Ke^{-\delta(t-T)})\varphi_{c_0}'(s)\\
 &&\;\;\;+(1+Ke^{-\delta(t-T)})H(\psi_{c_0}(s))-H((1+Ke^{-\delta(t-T)})\psi_{c_0}(s)).
 \eess
 For $s\in[0,1]$, we have
 \bess
 &&-\delta Ke^{-\delta(t-T)}\varphi_{c_0}(s)+\sigma\delta e^{-\delta(t-T)}(1+Ke^{-\delta(t-T)})\varphi_{c_0}'(s)\\
 &&+(1+Ke^{-\delta(t-T)})H(\psi_{c_0}(s))-H((1+Ke^{-\delta(t-T)})\psi_{c_0}(s))\\
 &&\ge e^{-\delta(t-T)}[-\delta Ku^*+\sigma\delta \min_{s\in[0,1]}\varphi_{c_0}'(s)]\ge0,
 \eess
 if $\sigma\min_{s\in[0,1]}\varphi_{c_0}'(s)\ge Ku^*$.

Denote
\[B_{\psi_{c_0}}=\max_{z\in[\psi_{c_0}(1),2v^*]}\bigg(\frac{H(z)}{z}\bigg)'\]
Due to condition {\bf (H)}, we know $B_{\psi_{c_0}}<0$. For $s\ge1$, we have
\bess
&&-\delta Ke^{-\delta(t-T)}\varphi_{c_0}(s)+\sigma\delta e^{-\delta(t-T)}(1+Ke^{-\delta(t-T)})\varphi_{c_0}'(s)\\
 &&+(1+Ke^{-\delta(t-T)})H(\psi_{c_0}(s))-H((1+Ke^{-\delta(t-T)})\psi_{c_0}(s))\\
 &&\ge-\delta Ke^{-\delta(t-T)}\varphi_{c_0}(s)+(1+Ke^{-\delta(t-T)})\psi_{c_0}(s)\left[\frac{H(\psi_{c_0}(s))}{\psi_{c_0}(s)}-\frac{H((1+Ke^{-\delta(t-T)})\psi_{c_0}(s))}{(1+Ke^{-\delta(t-T)})\psi_{c_0}(s)}\right]\\
 &&\ge-\delta Ke^{-\delta(t-T)}u^*+(1+Ke^{-\delta(t-T)})\psi_{c_0}(s)\left[\frac{H(\psi_{c_0}(s))}{\psi_{c_0}(s)}-\frac{H((1+e^{-\delta(t-T)})\psi_{c_0}(s))}{(1+e^{-\delta(t-T)})\psi_{c_0}(s)}\right]\\
 &&\ge-\delta Ke^{-\delta(t-T)}u^*+(1+Ke^{-\delta(t-T)})\psi_{c_0}^2(s)(-B_{\psi_{c_0}})e^{-\delta(t-T)}\\
 &&\ge e^{-\delta(t-T)}(-\delta Ku^*+\psi_{c_0}^2(1)(-B_{\psi_{c_0}}))\ge0
\eess
provided that $-\psi_{c_0}^2(1)B_{\psi_{c_0}}\ge \delta Ku^*$. Thus the first inequality of \eqref{3.1} holds if $\sigma\min_{s\in[0,1]}\varphi_{c_0}'(s)\ge Ku^*$ and $-\psi_{c_0}^2(1)B_{\psi_{c_0}}\ge \delta Ku^*$.

Analogously, we can verify that the second inequality of \eqref{3.1} is valid if $\sigma\min_{s\in[0,1]}\psi_{c_0}'(s)\ge Kv^*$ and $-\varphi_{c_0}^2(1)B_{\varphi_{c_0}}\ge \delta Kv^*$ where
\[B_{\varphi_{c_0}}=\max_{z\in[\varphi_{c_0}(1),2u^*]}\bigg(\frac{G(z)}{z}\bigg)'.\]

In a word, we know that \eqref{3.1} holds for any $T>0$ and $X>0$ if $K\ge1$ is large enough such that \eqref{3.2} is valid, and
\bes\label{3.3}\begin{cases}
c_0K\le \sigma\delta, ~ ~ \sigma\dd\min_{s\in[0,1]}\varphi_{c_0}'(s)\ge Ku^*, ~ ~ -\psi_{c_0}^2(1)B_{\psi_{c_0}}\ge \delta Ku^*,\\
 \sigma\dd\min_{s\in[0,1]}\psi_{c_0}'(s)\ge Kv^*, ~ ~ -\varphi_{c_0}^2(1)B_{\varphi_{c_0}}\ge \delta Kv^*,
 \end{cases}
\ees
which can be guaranteed by choosing $K\ge1$ sufficiently large, $\delta$ small enough and $\sigma$ suitably large. The proof is complete.
\end{proof}

Then we derive a lower bound of $h(t)-c_0t$ when $\mathbb{B}[w]=w$.

\begin{lem}\label{l3.2}Let $(u,v,h)$ be the unique solution of \eqref{1.2} with $\mathbb{B}[w]=w$. Suppose $(d_1,a,H'(v^*))=(d_2,b,G'(u^*))$. Then there exists a positive constant $C$ such that
\[h(t)-c_0t\ge -C ~ ~ {\rm for ~ }t\ge0.\]
\end{lem}
\begin{proof}
For any given $c\in(0,c_0)$, we define
\bess
&\ud{h}(t)=c_0(t-T)+c(T+1)-(1-e^{-\sigma(t-T)}), ~ ~ \ud u(t,x)=(1-\ep e^{-\sigma(t-T)})\varphi_{c_0}(\ud h(t)-x),\\
&\ud v(t,x)=(1-\ep e^{-\sigma(t-T)})\psi_{c_0}(\ud h(t)-x),
\eess
where positive constants $\sigma<c_0$, $T$ and $\ep$ are determined later.

We next show that by choosing $(\sigma,T, \ep)$ suitably, there holds:
\bes\begin{cases}\label{3.4}
\ud u_t\le d_1\ud u_{xx}-a\ud u+H(\ud v), &t>T,~x\in(ct,\ud h(t)),\\
 \ud v_t\le d_2\ud v_{xx}-b\ud v+G(\ud u), &t>T,~x\in(ct,\ud h(t)),\\
\ud u(t,ct)\le u(t,ct), ~ \ud v(t,ct)\le v(t,ct), &t>T,\\
\ud u(t,\ud h(t))\le0, ~ \ud v(t,\ud h(t))\le0, &t>T,\\
\ud h'(t)\le-\mu_1 \ud u_x(t,\ud h(t))-\mu_2\ud v_x(t,\ud h(t)), & t>T,\\
\ud h(T)\le h(T), ~ \ud u(T,x)\le u(T,x), ~ \ud v(T,x)\le v(T,x),&cT\le x\le \ud h(T).
 \end{cases}
 \ees
Once we have done that, by a comparison method we have $h(t)\ge\ud h(t)$ and $(u(t,x),v(t,x))\ge(\ud u(t,x),\ud v(t,x))$ for $t\ge T$ and $x\in[ct,\ud h(t)]$, which clearly indicates the result as wanted.

Let us now begin with proving \eqref{3.4}. Simple computations shows that if $c_0\ep\le\sigma$, then
\bess
-\mu_1 \ud u_x(t,\ud h(t))-\mu_2\ud v_x(t,\ud h(t))=(1-\ep e^{-\sigma(t-T)})c_0\ge c_0-\delta e^{-\sigma(t-T)}=\ud h'(t).
\eess
Denote $s=\ud h(t)-x$. Direct calculations yield
\bess
&&\ud u_t-d_1\ud u_{xx}+a\ud u-H(\ud v)\\
&&=\ep\sigma e^{-\sigma(t-T)}\varphi_{c_0}(s)-\sigma e^{-\sigma(t-T)}(1-\ep e^{-\sigma(t-T)})\varphi_{c_0}'(s)\\
&&\;\;\;+(1-\ep e^{-\sigma(t-T)})H(\psi_{c_0}(s))-H((1-\ep e^{-\sigma(t-T)})\psi_{c_0}(s)).
\eess
For $s\in[0,1]$, we have that if $\ep\le\frac{\min_{s\in[0,1]}\varphi_{c_0}'(s)}{2\varphi_{c_0}(1)}$, then
\bess
&&\ep\sigma e^{-\sigma(t-T)}\varphi_{c_0}(s)-\delta e^{-\sigma(t-T)}(1-\ep e^{-\sigma(t-T)})\varphi_{c_0}'(s)\\
&&\;\;\;+(1-\ep e^{-\sigma(t-T)})H(\psi_{c_0}(s))-H((1-\ep e^{-\sigma(t-T)})\psi_{c_0}(s))\\
&&\le\ep\sigma e^{-\sigma(t-T)}\varphi_{c_0}(s)-\sigma e^{-\sigma(t-T)}(1-\ep e^{-\sigma(t-T)})\varphi_{c_0}'(s)\\
&&=\sigma e^{-\sigma(t-T)}[\ep\varphi_{c_0}(s)-(1-\ep e^{-\sigma(t-T)})\varphi_{c_0}'(s)]\\
&&\le\sigma e^{-\sigma(t-T)}(\ep\varphi_{c_0}(1)-\frac{1}{2}\min_{s\in[0,1]}\varphi_{c_0}'(s))\le0.
\eess

Define
\[C_{\psi_{c_0}}=\max_{z\in[\psi_{c_0}(1)/2,v^*]}\bigg(\frac{H(z)}{z}\bigg)'.\]
Clearly, by {\bf (H)}, $C_{\psi_{c_0}}<0$. Straightforward computations yield that if $\sigma\le\frac{-\psi_{c_0}^2(1)C_{\psi_{c_0}}}{2v^*}$, then
\bess
&&\ep\sigma e^{-\sigma(t-T)}\varphi_{c_0}(s)-\sigma e^{-\sigma(t-T)}(1-\ep e^{-\sigma(t-T)})\varphi_{c_0}'(s)\\
&&\;\;\;+(1-\ep e^{-\sigma(t-T)})H(\psi_{c_0}(s))-H((1-\ep e^{-\sigma(t-T)})\psi_{c_0}(s))\\
&&\le\ep\sigma e^{-\sigma(t-T)}\varphi_{c_0}(s)+(1-\ep e^{-\sigma(t-T)})H(\psi_{c_0}(s))-H((1-\ep e^{-\sigma(t-T)})\psi_{c_0}(s))\\
&&\le\ep\sigma e^{-\sigma(t-T)}\varphi_{c_0}(s)+(1-\ep e^{-\sigma(t-T)})\ep e^{-\sigma(t-T)}\psi_{c_0}^2(s)C_{\psi_{c_0}}\\
&&\le\ep e^{-\sigma(t-T)}(\sigma v^*+\frac{1}{2}\psi_{c_0}^2(1)C_{\psi_{c_0}})\le0.
\eess
Similarly, we can show the second inequality of \eqref{3.4} holds if
\[\ep\le\frac{\min_{s\in[0,1]}\psi_{c_0}'(s)}{2\psi_{c_0}(1)} ~ ~ {\rm and }~ ~ \sigma\le\frac{-\varphi_{c_0}^2(1)C_{\varphi_{c_0}}}{2u^*},\]
 where
\[C_{\varphi_{c_0}}=\max_{z\in[\varphi_{c_0}(1)/2,u^*]}\bigg(\frac{G(z)}{z}\bigg)'.\]

Moreover, by Lemma \ref{l2.2}, for any given $c\in(0,c_0)$, there are $M$, $\delta$ and $T_1$ such that
\[(u(t,ct),v(t,ct))\ge(u^*-Me^{-\delta t},v^*-Me^{-\delta t}) ~ ~  {\rm for}~t\ge T_1.\]
 Let $\sigma\le \delta$, $T>T_1$, $\min\{u^*,v^*\}\ep e^{\sigma T}\ge M$.
Then we have
\bess
&&\underline{u}(t,ct)\le u^*-u^*\ep e^{\sigma T}e^{-\delta t}\le u^*-Me^{-\delta t}\le u(t,ct),\\
&&\underline{v}(t,ct)\le v^*-v^*\ep e^{\sigma T}e^{-\delta t}\le v^*-Me^{-\delta t}\le v(t,ct).
\eess
Since spreading happens, in light of \eqref{1.5} we can further choose $T$ large enough if necessary such that $h(T)\ge c(T+1)=\underline{h}(T)$ and for $x\in[cT,\underline{h}(T)]$,
\bess
\underline{u}(T,x)\le u^*(1-\ep)\le u(T,x) ~ ~ {\rm and }~ ~ \underline{v}(T,x)\le v^*(1-\ep)\le v(T,x).
\eess
Thus \eqref{3.4} is proved. So the proof is finished.
\end{proof}

Now we turn to the lower bound of $h(t)-c_0t$ when $\mathbb{B}[w]=w_x$. Notice that here we do not need other extra assumptions on parameters of \eqref{1.2} and nonlinear term $(H,G)$.

\begin{lem}\label{l3.3}Let $(u,v,h)$ be the unique solution of \eqref{1.2} with $\mathbb{B}[w]=w_x$. Then there exists a positive constant $C$ such that
\[h(t)-c_0t\ge -C ~ ~ {\rm for ~ }t\ge0.\]
\end{lem}
\begin{proof}Define
\bess
&&\underline{h}(t)=c_0(t-T)+L-(1-e^{-\delta(t-T)}), \\
&&\underline{u}(t,x)=(1-\ep e^{-\delta(t-T)})[\varphi_{c_0}(\ud h(t)-x)+\varphi_{c_0}(\ud h(t)+x)-\varphi_{c_0}(2\ud h)],\\
&&\underline{v}(t,x)=(1-\ep e^{-\delta(t-T)})[\psi_{c_0}(\ud h(t)-x)+\psi_{c_0}(\ud h(t)+x)-\psi_{c_0}(2\ud h)],
\eess
where positive constants $T$, $L$, $\delta<c_0$ and $\ep$ are determined later.

{\bf Step 1.} In this step, we show that by choosing $L$ to be sufficiently large, there holds
\[(\ud u(t,x),\ud v(t,x))\ge(0,0), ~ ~ \forall ~ t\ge T, ~ 0\le x \le \ud h(t).\]
For $\ud h\ge L>1$ and $x\in[0,\ud h]$, we denote
\[f(\ud h, x)=\varphi_{c_0}(\ud h-x)+\varphi_{c_0}(\ud h+x)-\varphi_{c_0}(2\ud h).\]
Obviously, the one-side partial derivative $\partial_xf(\ud h,\ud h)=-\varphi_{c_0}'(0)+\varphi_{c_0}'(2\ud h)$. Since $\varphi_{c_0}'(0)>0$ and $\varphi_{c_0}'(x)\to0$ as $x\to\yy$, there exists a $L_1\gg1$ such that $\partial_xf(\ud h,\ud h)<-\frac{\varphi_{c_0}'(0)}{2}$ when $L>L_1$. Notice that $\partial_{x\ud h} f(\ud h,x)=-\varphi_{c_0}''(\ud h-x)+\varphi_{c_0}''(\ud h+x)$ and $\partial_{xx}f(\ud h,x)=\varphi_{c_0}''(\ud h-x)+\varphi_{c_0}''(\ud h+x)$. Since $\varphi_{c_0}''(x)$ is bounded in $x\in[0,\yy)$, we know $\partial_{x\ud h}f$ and $\partial_{xx}f$ are bounded in $(\ud h,x)\in[L,\yy)\times[0,\ud h]$, which implies that $\partial_{x}f$ is uniformly continuous for $(\ud h,x)\in[1,\yy)\times[0,\ud h]$. Thus there exists a small $\eta>0$ such that $\partial_x f(\ud h,x)<\frac{-\varphi_{c_0}'(0)}{4}$ for $\ud h\ge L$ and $x\in[\ud h-\eta,\ud h]$. It is worthy mentioning that $\eta$ depends only on $\varphi_{c_0}$. Together with $f(\ud h,\ud h)=0$, we have $f(\ud h,x)\ge0$ for $\ud h\ge L$ and $x\in[\ud h-\eta,\ud h]$.

For $\ud h\ge L$ and $x\in[0,\ud h-\eta]$, we have $f(\ud h, x)\ge\varphi_{c_0}(\eta)+\varphi_{c_0}(\ud h)-\varphi_{c_0}(2\ud h)\ge0$
provided that $L$ is large enough, say $L>L_2\ge L_1$. Hence this step is finished.

{\bf Step 2.} In this step, we show that with the aid of suitable choices of $L$, $T$, $\ep$ and $\delta$, there holds
\bes\begin{cases}\label{3.5}
\ud u_t\le d_1\ud u_{xx}-a\ud u+H(\ud v), &t>T,~x\in(0,\ud h(t)),\\
 \ud v_t\le d_2\ud v_{xx}-b\ud v+G(\ud u), &t>T,~x\in(0,\ud h(t)),\\
\ud u_x(t,0)=\ud v_x(t,0)=\ud u(t,\ud h)=\ud v(t,\ud h)=0, &t>T,\\
\ud h'(t)\le-\mu_1 \ud u_x(t,\ud h(t))-\mu_2\ud v_x(t,\ud h(t)), & t>T,\\
\ud h(T)\le h(T), ~ \ud u(T,x)\le u(T,x), ~ \ud v(T,x)\le v(T,x),&0\le x\le \ud h(T).
 \end{cases}
 \ees
 Once it is proved, by a comparison argument, we directly obtain the desired result.

 Now let us prove \eqref{3.5}. The inequalities in the third line of \eqref{3.5} are obvious. By virtue of Lemma \ref{l2.1}, a straightforward computation yields
 \bess
 &&-\mu_1\ud u_x(t,\ud h(t))-\mu_2\ud v_x(t,\ud h(t))\\
 &&=c_0(1-\ep e^{-\delta(t-T)})-(1-\ep e^{-\delta(t-T)})(\mu_1\varphi_{c_0}'(2\ud h(t))+\mu_2\psi_{c_0}'(2\ud h(t)))\\
 &&\ge c_0-c_0\ep e^{-\delta(t-T)}-Ce^{-2\alpha \ud h(t)}\\
 &&\ge c_0-c_0\ep e^{-\delta(t-T)}-Ce^{2\alpha(1-L)}e^{-2\alpha c_0(t-T)}\\
 &&\ge c_0-(c_0\ep+Ce^{2\alpha(1-L)})e^{-\delta(t-T)}\\
 &&\ge c_0-\delta e^{-\delta(t-T)}=\ud h'(t)
\eess
provided that $\delta<2\alpha c_0$ and $c_0\ep+Ce^{2\alpha(1-L)}<\delta$, which are easily guaranteed by letting $\delta$, $\ep$ small enough and $L$ suitably large.
Thus the inequality in the fourth line of \eqref{3.5} holds.

Now we focus on the inequalities in the first two lines of \eqref{3.5}. Since the arguments used below are similar, we only prove the first one. To save spaces, we denote $\ep(t)=\ep e^{-\delta(t-T)}$. Direct computations yield
\bess
&&\ud u_t-d_1\ud u_{xx}-f_1(\ud u,\ud v)\\
&&=\delta \ep(t)\left[\varphi_{c_0}(\ud h(t)-x)+\varphi_{c_0}(\ud h(t)+x)-\varphi_{c_0}(2\ud h(t))\right]\\
&&~ ~ +(1-\ep(t))\left[f_1(\varphi_{c_0}(\ud h(t)-x),\psi_{c_0}(\ud h(t)-x))+f_1(\varphi_{c_0}(\ud h(t)+x),\psi_{c_0}(\ud h(t)+x))\right]-f_1(\ud u,\ud v)\\
&&~ ~ -(1-\ep (t))\delta\left[\varphi_{c_0}'(\ud h(t)-x)+\varphi_{c_0}'(\ud h(t)+x)\right]\\
&& ~ ~ -(1-\ep(t))(c_0-\delta e^{-\delta(t-T)})2\varphi_{c_0}'(2\ud h(t))
\eess
for $t>T$ and $x\in(0,\ud h(t))$. In view of the properties of the solution $(\varphi_{c_0},\psi_{c_0})$ of semi-wave problem \eqref{1.4}, we have
\[\delta \ep(t)\left[\varphi_{c_0}(\ud h(t)-x)+\varphi_{c_0}(\ud h(t)+x)-\varphi_{c_0}(2\ud h(t))\right]\le\delta\ep(t)u^*.\]
Due to our choices of $\delta$ and $\ep$, we see $-(1-\ep(t))(c_0-\delta e^{-\delta(t-T)})2\varphi_{c_0}'(2\ud h(t))\le0$.
Moreover, it follows from the assumption {\bf (H)} and Lemma \ref{l2.1} that
\bess
&&-f_1(\ud u,\ud v)\\
&&\le -f_1((1-\ep(t))[\varphi_{c_0}(\ud h-x)+\varphi_{c_0}(\ud h+x)-u^*], (1-\ep(t))[\psi_{c_0}(\ud h-x)+\psi_{c_0}(\ud h+x)-v^*])+C e^{-2\alpha \ud h},
\eess
where positive constant $C$ depends only on the parameters in \eqref{1.4} and nonlinear term $(H,G)$. Thus for $t>T$ and $x\in(0,\ud h(t))$, we obtain
\[\ud u_t-d_1\ud u_{xx}-f_1(\ud u,\ud v)\le \delta\ep(t)u^*+C e^{-2\alpha \ud h}+A(t,x)+B(t,x),\]
where
\bess &&A(t,x):=(1-\ep(t))\left[f_1(\varphi_{c_0}(\ud h(t)-x),\psi_{c_0}(\ud h(t)-x))+f_1(\varphi_{c_0}(\ud h(t)+x),\psi_{c_0}(\ud h(t)+x))\right]\\
&& ~ -f_1((1-\ep(t))[\varphi_{c_0}(\ud h(t)-x)+\varphi_{c_0}(\ud h(t)+x)-u^*], (1-\ep(t))[\psi_{c_0}(\ud h(t)-x)+\psi_{c_0}(\ud h(t)+x)-v^*])\\
&&B(t,x):=-(1-\ep (t))\delta\left[\varphi_{c_0}'(\ud h(t)-x)+\varphi_{c_0}'(\ud h(t)+x)\right].
\eess
For $\delta_0>0$ determined in Lemma \ref{l2.3}, we choose $K_0>0$ to be sufficiently large  such that
\[u^*-\varphi_{c_0}(K_0)\le\delta_0u^*{\rm ~ ~ and}~ ~ v^*-\psi_{c_0}(K_0)\le\delta_0v^*.\]
Then we let $L\gg K_0$ and estimate $\delta\ep(t)u^*+C e^{-2\alpha \ud h}+A(t,x)+B(t,x)$ in the following two cases:
\[x\in[0,\ud h(t)-K_0] ~ ~ {\rm and }~ ~ x\in[\ud h(t)-K_0,\ud h(t)].\]

First, for $t\ge T$ and $x\in[0,\ud h(t)-K_0]$, it is easy to see that
\[\varphi_{c_0}(\ud h(t)-x), ~ \varphi_{c_0}(\ud h(t)+x)\in[(1-\delta_0)u^*,u^*] {\rm ~ and ~ }\psi_{c_0}(\ud h(t)-x), ~ \psi_{c_0}(\ud h(t)+x)\in[(1-\delta_0)v^*,v^*].\]
Since $\ud h(t)+x\ge \ud h(t)$, it follows from Lemma \ref{l2.1} that
\bess
&&(u^*-\varphi_{c_0}(\ud h-x))(u^*-\varphi_{c_0}(\ud h+x))+(u^*-\varphi_{c_0}(\ud h-x))(v^*-\psi_{c_0}(\ud h+x))\\
&&~ ~ +(v^*-\psi_{c_0}(\ud h-x))(u^*-\varphi_{c_0}(\ud h+x))+(v^*-\psi_{c_0}(\ud h-x))(v^*-\psi_{c_0}(\ud h+x))\\
&&\le C\delta_0e^{-\alpha\ud h(t)}\le C\delta_0e^{-\alpha c_0(t-T))}e^{\alpha(1-L)}\le\delta_0\ep(t)
\eess
provided that $L$ is large enough such that $e^{\alpha(1-L)}\le \ep$. Hence we can apply Lemma \ref{l2.3} to derive that $A(t,x)\le -\sigma_0\ep(t)$ where $\sigma_0$ satisfies
\[-\sigma_0>\frac{1}{2}\max\{-au^*+H'(v^*)v^*, -bv^*+G'(u^*)\}.\]
Note that $B(t,x)\le0$ and $\delta<2\alpha c_0$. We have that for $t\ge T$ and $x\in[0,\ud h-K_0]$,
\bess
\delta\ep(t)u^*+C e^{-2\alpha \ud h}+A(t,x)+B(t,x)&\le& (\delta u^*-\sigma_0)\ep(t)+Ce^{2\alpha(1-L)}e^{-2\alpha c_0(t-T)}\\
&\le&\left[Ce^{2\alpha(1-L)}+(\delta u^*-\sigma_0)\ep\right]e^{-\delta(t-T)}<0,
\eess
if $Ce^{2\alpha(1-L)}+(\delta u^*-\sigma_0)\ep<0$ by letting $L$ large enough and $\delta$ suitably small.

For $t\ge T$ and $x\in[\ud h-K_0,\ud h]$, we easily derive that
\[B(t,x)\le -\frac{1}{2}\delta\min_{s\in[0,K_0]}\varphi_{c_0}'(s)e^{-\delta(t-T)},\]
where $\min_{s\in[0,K_0]}\varphi_{c_0}'(s)>0$.
Moreover, using Lemma \ref{l2.1} and arguing as in the proof of \cite[Lemma 2.9]{DN3}, we have $A(t,x)\le C(\ep(t)+e^{-\alpha\ud h(t)})$.
Thus for $t\ge T$ and $x\in[\ud h-K_0,\ud h]$, we can derive that
\bess
&&\delta\ep(t)u^*+C e^{-2\alpha \ud h}+A(t,x)+B(t,x)\\
&&\le\left[C e^{\alpha(1-L)}+C\ep-\frac{\delta}{2}\min_{s\in[0,K_0]}\varphi_{c_0}'(s)\right]e^{-\delta(t-T)}<0
\eess
provided that
\[C e^{\alpha(1-L)}+C\ep-\frac{\delta}{2}\min_{s\in[0,K_0]}\varphi_{c_0}'(s)<0.\]
Thus the first inequality of \eqref{3.5} holds with the above suitable choices of $L$, $\ep$ and $\delta$. Analogously, we can prove the second inequality of \eqref{3.5} is true if choosing $L$, $\ep$ and $\delta$ as above.

With $(L,\ep,\delta)$ determined as above, since spreading happens, it is easy to show that there exists a large $T$ such that $h(T)>L=\ud h(T)$ and $(u(T,x),v(T,x))\ge(\ud u(T,x),\ud v(T,x))$ for $x\in[0,\ud h(T)]$. So \eqref{3.5} holds, and the proof is finished.
\end{proof}

Clearly, Proposition \ref{p3.1} follows from Lemmas \ref{l3.1}-\ref{l3.3}.

\section{The proof of Theorem \ref{t1.1}}
In this section, we will complete the proof of Theorem \ref{t1.1} by following the similar line as in \cite{WND1}. According to Proposition \ref{p3.1}, there exists a $\tilde{C}>0$ such that
\[-\tilde{C}\le h(t)-c_0t\le\tilde{C} ~ ~ {\rm for ~ }t\ge0.\]
Define
\bess
k(t)=c_0t-2\tilde{C}, ~ l(t)=h(t)-k(t), ~ w(t,x)=u(t,x+k(t)), ~ z(t,x)=v(t,x+k(t)).
\eess
Clearly, $\tilde{C}\le l(t)\le 3\tilde{C}$ for $t\ge0$, and $(w,z)$ satisfies
\bess\begin{cases}
w_t=d_1w_{xx}+c_0w_x-aw+H(z), &t>0,~x\in(-k(t),l(t)),\\
z_t=d_2z_{xx}+c_0z_x-bz+G(w), &t>0,~x\in(-k(t),l(t)),\\
w(t,l(t))=z(t,l(t))=0, &t>0,\\
l'(t)=-\mu_1w_x(t,l(t))-\mu_2z_x(t,l(t))-c_0, &t>0.
 \end{cases}
 \eess
Choose a sequence $\{t_n\}$ with $t_n>0$, $t_n\to\yy$ and $\lim_{n\to\yy}l(t_n)=\liminf_{t\to\yy}l(t)\ge \tilde{C}>0$. Define
\bess
(k_n(t),l_n(t))=(k(t+t_n),l(t+t_n)) {\rm ~ ~ and ~ ~ }(w_n(t,x),z_n(t,x))=(w(t+t_n,x),z(t+t_n,x))
\eess
with $t>-t_n$ and $x\in(-k_n(t),l_n(t))$.

\begin{lem}\label{l4.1}There exists a subsequence of $\{t_n\}$, still denoted by itself, such that as $n\to\yy$,
\bess
l_n\to L ~ {\rm in ~ }C^{1+\frac{1+\alpha}{2}}_{{\rm loc}}(\mathbb{R}) {\rm ~ ~ and ~ ~ }(w_n,z_n)\to(W,Z) {\rm ~ in ~ }C^{1+\frac{\alpha}{2},2+\alpha}_{{\rm loc}}(\Omega),
\eess
where $\alpha\in(0,1)$ and $\Omega:=\{(t,x):t\in\mathbb{R}, ~ -\yy<x<L(t)\}$. Moreover, $(W,Z,L)$ satisfies
\bes\label{4.1}\begin{cases}
W_t=d_1W_{xx}+c_0W_x-aW+H(Z), &t\in\mathbb{R},~x\in(-\yy,L(t)),\\
Z_t=d_2Z_{xx}+c_0Z_x-bZ+G(W), &t\in\mathbb{R},~x\in(-\yy,L(t)),\\
W(t,L(t))=Z(t,L(t))=0, &t\in\mathbb{R},\\
L'(t)=-\mu_1W_x(t,L(t))-\mu_2Z_x(t,L(t))-c_0, ~ L(t)\ge L(0)\ge \tilde{C}, &t\in\mathbb{R}.
 \end{cases}
 \ees
\end{lem}

\begin{proof}
By virtue of Lemma \ref{l2.0}, we see that $0<h'(t)\le C$ for $t>0$, which implies that
\bes\label{4.2}-c_0<l'_n(t)\le C-c_0 ~ ~ {\rm for ~ }t>-t_n.\ees
Denote
\[\xi=\frac{x}{l_n(t)} ~ ~ {\rm and ~ ~ }(\tilde{w}_n(t,\xi),\tilde{z}_n(t,\xi))=(w_n(t,\xi l_n(t)),z_n(t,\xi l_n(t))),\]
where $t>-t_n$ and $-k_n(t)/l_n(t)<\xi<1$. Clearly, $(\tilde{w}_n,\tilde{z}_n, l_n)$ satisfies
\bes\label{4.3}\begin{cases}
\tilde{w}_{nt}=\frac{d_1}{l^2_n(t)}\tilde w_{n\xi\xi}+\frac{\xi l'_n(t)+c_0}{l_n(t)}\tilde w_{n\xi}-a\tilde w_n+H(\tilde z_n), &t>-t_n,~\xi\in(-\frac{k_n(t)}{l_n(t)},1),\\
\tilde{z}_{nt}=\frac{d_2}{l^2_n(t)}\tilde{z}_{n\xi\xi}+\frac{\xi l'_n(t)+c_0}{l_n(t)}\tilde z_{n\xi}-b\tilde z_n+G(\tilde w_n), &t>-t_n,~\xi\in(-\frac{k_n(t)}{l_n(t)},1),\\
\tilde{w}_n(t,1)=\tilde{z}_n(t,1)=0, &t>-t_n,\\
l'_n(t)=\frac{-\mu_1}{l_n(t)}\tilde{w}_{n\xi}(t,1)-\frac{\mu_2}{l_n(t)}\tilde{z}_{n\xi}(t,1)-c_0, &t>-t_n.
 \end{cases}
 \ees
In light of Lemma \ref{l2.0} and \eqref{4.2}, for any $R,T\in\mathbb{R}$, we can apply the local $L^p$ estimates (see e.g. Theorem 1.10 in \cite{Wpara}, pp 10) to the equations of $\tilde{w}_n$ and $\tilde{z}_n$ in \eqref{4.3} over the domain $[-R-2,1]\times[T-3,T+1]$, respectively, and derive that for any $p>1$, there holds
\[\|\tilde{w}_n,\tilde{z}_n\|_{W^{1,2}_p([-R-1,1]\times[T-2,T+1])}\le C_1 ~ ~{\rm for ~ all ~ large ~ }n,\]
where $C_1$ depends only on the parameters of \eqref{1.2}, the initial functions $u_0$ and $v_0$, $R$ and $p$, but not on $n$ and $T$. Moreover, for any $\alpha'\in(0,1)$, we can take $p>1$ large enough and use the Sobolev embedding theorem to get
\bes
\|\tilde{w}_n,\tilde{z}_n\|_{C^{\frac{1+\alpha'}{2},1+\alpha'}([-R-1,1]\times[T-2,T+1])}\le C_2 ~ ~{\rm for ~ all ~ large ~ }n,
\ees
where $C_2$ depends only on the parameters of \eqref{1.2}, the initial functions $u_0$ and $v_0$, $R$ and $p$, but not on $n$ and $T$. This obviously implies that
\[\|l_n\|_{C^{1+\frac{\alpha'}{2}}([T-2,T+1])}\le C_3 ~ ~{\rm for ~ all ~ large ~ }n,\]
where $C_3$ has the same dependence as $C_1$ or $C_2$. Thus we can apply the local Schauder estimates (see e.g. Theorem 1.12 in \cite{Wpara}, pp 13) to the equations of $\tilde{w}_n$ and $\tilde{z}_n$ in \eqref{4.3} over the domain $[-R-1,1]\times[T-2,T+1]$, respectively, and derive that
\[\|\tilde{w}_n,\tilde{z}_n\|_{C^{1+\frac{\alpha'}{2},2+\alpha'}([-R,1]\times[T-1,T+1])}\le C_4 ~ ~{\rm for ~ all ~ large ~ }n,\]
where $C_4$ also relies only on the parameters of \eqref{1.2}, the initial functions $u_0$ and $v_0$, $R$ and $p$, but not on $n$ and $T$, which indicates
\[\|l_n\|_{C^{1+\frac{1+\alpha'}{2}}([T,\yy))}\le C_5  ~ ~{\rm for ~ all ~ large ~ }n,\]
where $C_5$ has the same dependence as $C_1$ or $C_2$. Then for any $\alpha\in(0,\alpha')$, owing to a compact consideration and the arbitrariness of $(R,T)$, we have that by passing to a subsequence, still denoted by itself, we have that as $n\to\yy$,
\[l_n\to L ~ {\rm in ~ }C^{1+\frac{1+\alpha}{2}}_{\rm loc}(\mathbb{R}) ~ ~ {\rm and ~ ~ }(\tilde{w}_n,\tilde{z}_n)\to(\tilde W,\tilde Z) ~ {\rm in ~ }C^{1+\frac{\alpha}{2},2+\alpha}_{\rm loc}((-\yy,1]\times\mathbb{R}),\]
where $(\tilde{W},\tilde{Z},L)$ satisfies
\bess\begin{cases}
\tilde{W}_{t}=\frac{d_1}{L^2(t)}\tilde W_{\xi\xi}+\frac{\xi L'(t)+c_0}{L(t)}\tilde W_{\xi}-a\tilde W+H(\tilde Z), &t\in\mathbb{R},~\xi\in(-\yy,1),\\
\tilde{Z}_{t}=\frac{d_2}{L^2(t)}\tilde{Z}_{\xi\xi}+\frac{\xi L'(t)+c_0}{L(t)}\tilde Z_{\xi}-b\tilde Z+G(\tilde W), &t\in\mathbb{R},~\xi\in(-\yy,1),\\
\tilde{W}(t,1)=\tilde{Z}(t,1)=0, &t\in\mathbb{R},\\
L'(t)=\frac{-\mu_1}{L(t)}\tilde{W}_{\xi}(t,1)-\frac{\mu_2}{L(t)}\tilde{Z}_{\xi}(t,1)-c_0, &t\in\mathbb{R}.
 \end{cases}
 \eess
Let $(W(t,x),Z(t,x))=(\tilde{W}(t,\frac{x}{L(t)}),\tilde{Z}(t,\frac{x}{L(t)}))$. Then it is easy to see that $(W,Z,L)$ satisfies \eqref{4.1}. Besides, $(w_n,z_n)\to(W,Z)$ in $C^{1+\frac{\alpha}{2},2+\alpha}_{{\rm loc}}(\Omega)$.
Notice that $L(0)=\lim_{n\to\yy}l(t_n)=\liminf_{t\to\yy}l(t)\le \lim_{n\to\yy}l(t+t_n)=L(t)$. So $L(t)\ge L(0)$ for all $t\in\mathbb{R}$. Thus the proof is complete.
\end{proof}

We next study the properties of $(W,Z,L)$. By virtue of strong maximum principle, we know that $(W(t,x),Z(t,x))>(0,0)$ for $x<L(t)$ and $t\in\mathbb{R}$.

\begin{lem}\label{l4.2} $L(t)\equiv L(0)$ and $(W(t,x),Z(t,x))=(\varphi_{c_0}(L(0)-x),\psi_{c_0}(L(0)-x))$,
where $(\varphi_{c_0},\psi_{c_0})$ is the unique monotone solution of \eqref{1.4}.
\end{lem}

\begin{proof}This proof will be divided into several steps.

{\bf Step 1.} In this step, we will show that there exists $C_1<L(0)$ such that
\[(W(t,x),Z(t,x))\ge(\varphi_{c_0}(C_1-x),\psi_{c_0}(C_1-x)) ~ ~ {\rm for ~ }t\in\mathbb{R}, ~ x\le C_1.\]
Taking advantage of Lemma \ref{l3.2} and \ref{l3.3}, we have the following statements.
\begin{enumerate}[(i)]
  \item If $\mathbb{B}[w]=w$, then for any $t+t_n>T$ and $c(t+t_n)-k(t+t_n)<x<\ud h(t+t_n)-k(t+t_n)$,
  \bess
  \begin{cases}
    w_n(t,x)\ge(1-\ep e^{-\sigma(t+t_n-T)})\varphi_{c_0}(\ud h(t+t_n)-k(t+t_n)-x),\\
    z_n(t,x)\ge(1-\ep e^{-\sigma(t+t_n-T)})\psi_{c_0}(\ud h(t+t_n)-k(t+t_n)-x),
  \end{cases}
  \eess
  where $T$ and $\ud h$ are determined in Lemma \ref{l3.2}.
  \item If $\mathbb{B}[w]=w_x$, then for any $t+t_n>T$ and $-k(t+t_n)<x<\ud h(t+t_n)-k(t+t_n)$,
  \bess
  \begin{cases}
    w_n(t,x)\ge(1-\ep e^{-\delta(\tau_n-T)})\left[\varphi_{c_0}(\ud h(\tau_n)-k(\tau_n)-x)+\varphi_{c_0}(\ud h(\tau_n)+k(\tau_n)+x)-\varphi_{c_0}(2\ud h(\tau_n))\right],\\
    z_n(t,x)\ge(1-\ep e^{-\delta(\tau_n-T)})\left[\psi_{c_0}(\ud h(\tau_n)-k(\tau_n)-x)+\psi_{c_0}(\ud h(\tau_n)+k(\tau_n)+x)-\psi_{c_0}(2\ud h(\tau_n))\right],
  \end{cases}
  \eess
  where $\tau_n=t+t_n$, $T$ and $\ud h$ are determined in Lemma \ref{l3.3}.
\end{enumerate}
It is easy to see that there exists a $C_1\in\mathbb{R}$ such that $C_1<L(0)$ and $h(t+t_n)-k(t+t_n)\ge C_1$ for $t+t_n>T$. Hence, for any $x\le C_1$ and $t\in\mathbb{R}$,
letting $n\to\yy$, we have
\[(W(t,x),Z(t,x))\ge(\varphi_{c_0}(C_1-x),\psi_{c_0}(C_1-x)) ~ ~ {\rm for ~ }t\in\mathbb{R},~ x\le C_1,\]
which completes this step.

{\bf Step 2.}
Due to Step 1, we can define
\[C^*:=\sup\{C: (W(t,x),Z(t,x))\ge (\varphi_{c_0}(C-x),\psi_{c_0}(C-x)) ~ {\rm for ~ }t\in\mathbb{R}, ~ x\le C \}.\]
Obviously, $C^*\le L(0)$ and $(W(t,x),Z(t,x))\ge (\varphi_{c_0}(C^*-x),\psi_{c_0}(C^*-x))$  for $x\le C^*$ and $x\in\mathbb{R}$. In this step, we will show $C^*=L(0)$. Assume on the contrary that $C^*<L(0)$.

{\bf Claim 1.} $(W(t,x),Z(t,x))>(\varphi_{c_0}(C^*-x),\psi_{c_0}(C^*-x))$  for $x\le C^*$ and $x\in\mathbb{R}$.

Otherwise, since $C^*<L(0)=\min_{t\in\mathbb{R}}\{L(t)\}$ and $(W(t,C^*),Z(t,C^*))>(0,0)$, there exists $(t_0,x_0)\in(-\yy,C^*)\times\mathbb{R}$ such that $W(t_0,x_0)=\varphi_{c_0}(C^*-x_0)$ or $Z(t_0,x_0)=\psi_{c_0}(C^*-x_0)$. Without loss of generality, we suppose that $W(t_0,x_0)=\varphi_{c_0}(C^*-x_0)$. Clearly, $(\varphi_{c_0}(C^*-x),\psi_{c_0}(C^*-x))$ satisfies the first two equations of \eqref{4.1} for $(-\yy,C^*)\times\mathbb{R}$.  Thus we see that for $x<C^*$ and $t\in\mathbb{R}$,
\[(W(t,x)-\varphi_{c_0}(C^*-x))_t\ge d_1(W(t,x)-\varphi_{c_0}(C^*-x))_{xx}+c_0(W(t,x)-\varphi_{c_0}(C^*-x))_x-a(W(t,x)-\varphi_{c_0}(C^*-x)).\]
In view of strong maximum principle, we derive that $W(t,x)-\varphi_{c_0}(C^*-x)\equiv0$ for $x<C^*$ and $t\in\mathbb{R}$. However,  since $W(t,C^*)-\varphi_{c_0}(C^*-C^*)=W(t,C^*)>0$ by $C^*<L(0)$, we obtain a contradiction. Thus our claim holds.

{\bf Claim 2.} For any $x\le C^*$,
\bess
\begin{cases}
  \omega_1(x):=\sup_{t\in\mathbb{R},y\in[x,C^*]}(\varphi_{c_0}(C^*-y)-W(t,y))<0,\\
  \omega_2(x):=\sup_{t\in\mathbb{R},y\in[x,C^*]}(\psi_{c_0}(C^*-y)-Z(t,y))<0.
\end{cases}
\eess
In light of Claim 1, we know $w_i(x)\le0$ for $x\le C^*$ and $i=1,2$. Arguing indirectly, we have $w_1(x_0)=0$ or $w_2(x_0)=0$ for some $x_0<C^*$. Without loss of generality, we may assume that $w_1(x_0)=0$. From Claim 1, we know the supremum $w_1(x_0)$ cannot be achieved at any $(t,y)\in[x_0,C^*]\times\mathbb{R}$. Thus there exists a sequence $\{(s_n,y_n)\}\subset[x_0,C^*]\times\mathbb{R}$ with $|s_n|\to\yy$ such that
\bes\label{4.5}\varphi_{c_0}(C^*-y_n)-W(s_n,y_n)\to0 ~ ~ {\rm as ~ }n\to\yy.\ees
By passing to a subsequence, we can assume that $y_n\to y_0\in[x_0,C^*]$. Define
\[(W_n(t,x),Z_n(t,x),L_n(t))=(W(t+s_n,x+y_n),Z(t+s_n,x+y_n),L(t+s_n)-y_n)\]
with $t\in\mathbb{R}$ and $x\le L_n(t)$. Due to $l\to L$ in $C^{1+\frac{1+\alpha}{2}}_{{\rm loc}}(\mathbb{R})$ and \eqref{4.2}, we know that $L'_n$ is uniformly bounded for $n$ and $t\in\mathbb{R}$. Moreover, $0<L(0)-C^*\le L_n(t)\le 3C-x_0$ for $t\in\mathbb{R}$. So we can argue as in the proof of Lemma \ref{l4.1} to deduce that by passing to a subsequence if necessary, for $\alpha\in(0,1)$ there holds
\[(W_n,Z_n,L_n)\to(\tilde{W},\tilde{Z},\tilde{L}){\rm ~ and ~ }[C^{1+\frac{\alpha}{2},2+\alpha}_{{\rm loc}}(\tilde{\Omega})]^2\times C^{1+\frac{1+\alpha}{2}}_{{\rm loc}}(\mathbb{R}){\rm ~ as ~ }n\to\yy,\]
where $\tilde{\Omega}:=\{(t,x):t\in\mathbb{R}, ~ x<\tilde{L}(t)\}$, and $(\tilde{W},\tilde{Z},\tilde{L})$ satisfies
\bes\label{4.6}\begin{cases}
\tilde W_t=d_1\tilde W_{xx}+c_0\tilde W_x-a\tilde W+H(\tilde Z), &t\in\mathbb{R},~x\in(-\yy,\tilde L(t)),\\
\tilde Z_t=d_2\tilde Z_{xx}+c_0\tilde Z_x-b\tilde Z+G(\tilde W), &t\in\mathbb{R},~x\in(-\yy,\tilde L(t)),\\
\tilde W(t,\tilde L(t))=\tilde Z(t,\tilde L(t))=0, &t\in\mathbb{R}.
 \end{cases}
 \ees
 Furthermore, $\tilde{L}(t)>C^*-y_0$ for $t\in\mathbb{R}$, $(\tilde{W}(t,x),\tilde{Z}(t,x))>(0,0)$ for $(t,x)\in\mathbb{R}\times(-\yy,\tilde{L}(t))$, and
 \[(\tilde{W}(t,x),\tilde{Z}(t,x))\ge(\varphi_{c_0}(C^*-y_0-x),\psi_{c_0}(C^*-y_0-x)) ~ ~ {\rm for ~ }t\in\mathbb{R}, ~ x<C^*-y_0.\]
 In view of \eqref{4.5}, we see $\tilde{W}(0,0)=\varphi_{c_0}(C^*-y_0)$.

 On the other hand, it is easy to see that $(\varphi_{c_0}(C^*-y_0-x),\psi_{c_0}(C^*-y_0-x))$ satisfies \eqref{4.6} with $C^*-y_0$ in place of $\tilde{L}(t)$. Then by the strong maximum principle, we can derive that $\tilde{W}(t,x)\equiv\varphi_{c_0}(C^*-y_0-x)$ for $t\le0$ and $x<C^*-y_0$, which implies that $\tilde{W}(0,C^*-y_0)=0$. However, since $\tilde{L}(t)>C^*-y_0$, we have $\tilde{W}(t,C^*-y_0)>0$. This contradiction completes the proof of this claim.

 Since $(\varphi_{c_0}(C^*-x),\psi_{c_0}(C^*-x))\to(u^*,v^*)$ as $x\to-\yy$, for any small $\ep_0>0$ there exists $X_0<C^*$ such that
 \[(\varphi_{c_0}(C^*-x),\psi_{c_0}(C^*-x))\ge(u^*-\ep_0,v^*-\ep_0) ~ ~ {\rm for ~ }x\le X_0.\]
 Let $\ep\in(0,\ep_0)$ such that
 \[(\varphi_{c_0}(C^*-X_0),\psi_{c_0}(C^*-X_0))-(\omega_1(X_0),\omega_2(X_0))\ge(\varphi_{c_0}(C^*-X_0+\ep),\psi_{c_0}(C^*-X_0+\ep)).\]
 Then we consider the following auxiliary problem
 \bes\label{4.7}
 \begin{cases}
   \bar{W}_t=d_1\bar{W}_{xx}+c_0\bar{W}_x-a\bar{W}+H(\bar{Z}), &t>0, ~ x<X_0, \\
   \bar{Z}_t=d_2\bar{Z}_{xx}+c_0\bar{Z}_x-b\bar{Z}+G(\bar{W}), &t>0, ~ x<X_0, \\
   \bar{W}(t,X_0)=\varphi_{c_0}(C^*-X_0+\ep), ~ \bar{Z}(t,X_0)=\psi_{c_0}(C^*-X_0+\ep), & t>0, \\
   \bar{W}(0,x)=\varphi_{c_0}(C^*-x), ~ \bar{Z}(0,x)=\psi_{c_0}(C^*-x), &x<X_0.
 \end{cases}
\ees
Clearly, $(u^*,v^*)$ and $(\varphi_{c_0}(C^*-x),\psi_{c_0}(C^*-x))$ are a pair of upper and lower solutions of \eqref{4.7}. By a comparison argument, we see that $(\bar{W}(t,x),\bar{Z}(t,x))$ is nondecreasing in $t>0$, and
\bes\label{4.8}
(\varphi_{c_0}(C^*-x),\psi_{c_0}(C^*-x))\le(\bar{W}(t,x),\bar{Z}(t,x))\le(u^*,v^*) ~ ~ {\rm for ~ }t>0,~ x<X_0.
\ees
Moreover, we have
\[(\bar{W}(t,x),\bar{Z}(t,x))\to(W^*(x),Z^*(x)) ~ ~ {\rm in ~ }[C^{2}_{{\rm loc}}((-\yy,X_0)),\]
where $(W^*,Z^*)$ satisfies
 \bes\label{4.9}
\left\{\!\begin{aligned}
&d_1{W^*}''+c_0{W^*}'-aW^*+H(Z^*)=0, &x<X_0,\\
&d_2{Z^*}''+c_0{Z^*}'-bZ^*+G(W^*)=0, &x<X_0,\\
&W^*(-\yy)=u^*, ~ Z^*(-\yy)=v^*,\\
&W^*(X_0)=\varphi_{c_0}(C^*-X_0+\ep), ~ Z^*(X_0)=\psi_{c_0}(C^*-X_0+\ep).
\end{aligned}\right.
 \ees
It is easy to verify that $(\varphi_{c_0}(C^*-x+\ep),\psi_{c_0}(C^*-x+\ep))$ satisfies \eqref{4.9}. Together with $(\varphi_{c_0}(C^*-x+\ep),\psi_{c_0}(C^*-x+\ep))\ge(\varphi_{c_0}(C^*-x),\psi_{c_0}(C^*-x))$, by a comparison consideration, we have
\[(\varphi_{c_0}(C^*-x+\ep),\psi_{c_0}(C^*-x+\ep))\ge(\bar{W}(t,x),\bar{Z}(t,x)) ~ ~ {\rm for ~ }t>0, ~ x<X_0.\]
Letting $t\to\yy$ yields
\bes\label{4.11}(\varphi_{c_0}(C^*-x+\ep),\psi_{c_0}(C^*-x+\ep))\ge(W^*(x),Z^*(x)) ~ ~ {\rm for ~ }x\le X_0.\ees

{\bf Claim 3.} $(\varphi_{c_0}(C^*-x+\ep),\psi_{c_0}(C^*-x+\ep))=(W^*(x),Z^*(x))$ for $x\le X_0$.

Otherwise, according to $(\varphi_{c_0}(C^*-X_0+\ep),\psi_{c_0}(C^*-X_0+\ep))=(W^*(X_0),Z^*(X_0))$, $(\varphi_{c_0}(-\yy),\psi_{c_0}(-\yy))=(W^*(-\yy),Z^*(-\yy))=(u^*,v^*)$ and \eqref{4.11}, there holds
\bess &&W^*(x_1)-\varphi_{c_0}(C^*-x_1+\ep)=\min_{x\in(-\yy,X_0]}\{W^*(x)-\varphi_{c_0}(C^*-x+\ep)\}<0 ~ {\rm for ~ some ~ }x_1<X_0, ~ {\rm or}\\
&&Z^*(x_2)-\psi_{c_0}(C^*-x_2+\ep)=\min_{x\in(-\yy,X_0]}\{Z^*(x)-\psi_{c_0}(C^*-x+\ep)\}<0 ~ {\rm for ~ some ~ }x_2<X_0.
\eess
Without loss of generality, we assume $W^*(x_1)-\varphi_{c_0}(C^*-x_1+\ep)=\min_{x\in(-\yy,X_0]}\{W^*(x)-\varphi_{c_0}(C^*-x+\ep)\}<0$. For convenience, we denote $(\hat{W},\hat{Z})=(W^*(x)-\varphi_{c_0}(C^*-x+\ep),Z^*(x)-\psi_{c_0}(C^*-x+\ep))$. Simple calculations show that $(\hat{W},\hat{Z})$ satisfies
\bes
\begin{cases}
  d_1\hat{W}''+c_0\hat{W}'-a\hat{W}+H(Z^*)-H(\psi_{c_0}(C^*-x+\ep))=0, & x<X_0, \\
  d_2\hat{Z}''+c_0\hat{Z}'-b\hat{Z}+G(W^*)-G(\varphi_{c_0}(C^*-x+\ep))=0, & x<X_0,
\end{cases}
\ees
which further implies that
\bess
\begin{cases}
  a\hat{W}(x_1)-H(Z^*(x_1))+H(\psi_{c_0}(C^*-x_1+\ep))\ge0 ~ ~ {\rm and ~ ~ }\hat{Z}(x_1)<0,\\
  b\hat{Z}(x_2)-G(W^*(x_2))+G(\varphi_{c_0}(C^*-x_2+\ep))\ge0 ~ ~ {\rm and ~ ~ }\hat{W}(x_2)<0.
\end{cases}
\eess
Due to \eqref{4.8}, we have
\bes\label{4.10}
(W^*(x),Z^*(x))\ge(\varphi_{c_0}(C^*-x),\psi_{c_0}(C^*-x))\ge(u^*-\ep_0,v^*-\ep_0) ~ ~ {\rm for ~ }x\le X_0.
\ees
By the mean value theorem, there exist $\eta$ and $\zeta$ satisfying
\bess
\begin{cases}
  v^*-\ep_0\le Z^*(x_1)<\eta<\psi_{c_0}(C^*-x_1+\ep)<v^*,\\
  u^*-\ep_0\le W^*(x_2)<\zeta<\varphi_{c_0}(C^*-x_2+\ep)<u^*,
\end{cases}
\eess
such that
\bess
a\hat{W}(x_1)-H'(\eta)\hat{Z}(x_1)\ge0 ~ ~ {\rm and ~ ~ }b\hat{Z}(x_2)-G'(\zeta)\hat{W}(x_2)\ge0,
\eess
which, combined with $\hat{W}(x_1)\le\hat{W}(x_2)<0$ and $\hat{Z}(x_2)\le\hat{Z}(x_1)<0$, yields that $ab\le H'(\eta)G'(\zeta)$. However, since $ab>H'(v^*)G'(u^*)$, we can choose $\ep_0$ sufficiently small such that $ab>H'(v^*-\ep_0)G'(u^*-\ep_0)>H'(\eta)G'(\zeta)$. Then we derive a contradiction. Thus Claim 3 is true.

Noticing that for any $t\in\mathbb{R}$ and $x\le C^*$, we have
\bess
&&(W(t,x),Z(t,x))\ge(\varphi_{c_0}(C^*-x),\psi_{c_0}(C^*-x)),\\
&&W(t,X_0)\ge\varphi_{c_0}(C^*-X_0)-\omega_1(X_0)\ge\varphi_{c_0}(C^*-X_0+\ep),\\
&&Z(t,X_0)\ge\psi_{c_0}(C^*-X_0)-\omega_2(X_0)\ge\psi_{c_0}(C^*-X_0+\ep).
\eess
It then follows from a comparison method that
\[(W(t+s,x),Z(t+s,x))\ge(\bar{W}(t,x),\bar{Z}(t,x)) ~ ~ {\rm for ~ all}~ t>0, ~ x<X_0, ~ s\in\mathbb{R},\]
which clearly indicates
\[(W(t,x),Z(t,x))\ge(\bar{W}(t-s,x),\bar{Z}(t-s,x)) ~ ~ {\rm for ~ all}~ t>s, ~ x<X_0, ~ s\in\mathbb{R}.\]
Letting $s\to-\yy$ and using Claim 3, we have
\bes\label{4.13}
(W(t,x),Z(t,x))\ge(W^*(x),Z^*(x))=(\varphi_{c_0}(C^*-x+\ep),\psi_{c_0}(C^*-x+\ep)) ~ \forall ~ t\in\mathbb{R}, ~ x<X_0.
\ees

On the other hand, we denote $\rho=\min\{-\omega_1(X_0),-\omega_2(X_0)\}$. Owing to Claim 2, $\rho>0$. By continuity,we can find $\ep_1\in(0,\ep)$ small enough such that
\[(\varphi_{c_0}(C^*-x+\ep_1),\psi_{c_0}(C^*-x+\ep_1))\le (\varphi_{c_0}(C^*-x),\psi_{c_0}(C^*-x))+(\rho,\rho) ~ {\rm for ~ }x\in[X_0,C^*+\ep_1].\]
Therefore, for $t\in\mathbb{R}$ and $x\in[X_0,C^*+\ep_1]$, we see
\[(W(t,x),Z(t,x))-(\varphi_{c_0}(C^*-x+\ep_1),\psi_{c_0}(C^*-x+\ep_1))\ge-(\rho,\rho)-(\omega_1(X_0),\omega_2(X_0))\ge(0,0),\]
which, together with \eqref{4.13}, leads to
\[(W(t,x),Z(t,x))-(\varphi_{c_0}(C^*-x+\ep_1),\psi_{c_0}(C^*-x+\ep_1))\ge(0,0) ~ ~ {\rm for ~ }t\in\mathbb{R}, ~ x\le C^*+\ep_1.\]
This contradicts the definition of $C^*$. Thus $C^*=L(0)$. The Step 2 is finished.

{\bf Step 3.} This step will complete the proof of this lemma.

From the above analysis, we know
\bess
&C^*=L(0)=\dd\min_{t\in\mathbb{R}}L(t), ~ L'(0)=0,\\
&(W(t,x),Z(t,x))-(\varphi_{c_0}(C^*-x),\psi_{c_0}(C^*-x))\ge(0,0) ~ ~ {\rm for ~ }t\in\mathbb{R}, ~ x\le C^*,\\
&(W(0,L(0)),Z(0,L(0)))=(0,0), ~ W_x(0,L(0))\le -\varphi_{c_0}'(0), ~ Z_x(0,L(0))\le -\psi_{c_0}'(0).
\eess
Note that $\mu_1\varphi_{c_0}'(0)+\mu_2\psi_{c_0}'(0)=c_0$. We deduce $W_x(0,L(0))=-\varphi_{c_0}'(0)$ and $Z_x(0,L(0))=-\psi_{c_0}'(0)$. Then applying the strong maximum principle and the Hopf boundary lemma to \eqref{4.1} arrives at $(W(t,x),Z(t,x))=(\varphi_{c_0}(C^*-x),\psi_{c_0}(C^*-x))$ for $t\in\mathbb{R}$ and $x\le L(t)$, which obviously implies that $L(t)\equiv L(0)$ for $t\in\mathbb{R}$. The proof is ended.
\end{proof}

With the aid of the above results, we now finish the proof of Theorem \ref{t1.1}.

\begin{proof}[{\bf Proof of Theorem \ref{t1.1}:}]\, The proof is divided into three steps.

{\bf Step 1.} In this step, we will show that for any sequence $\{t_n\}$ defined in Lemma \ref{l4.1}, the following statements are valid.
\begin{enumerate}[(i)]
  \item For any $t\in\mathbb{R}$, $h'(t+t_n)\to c_0$ as $n\to\yy$.
  \item When $\mathbb{B}[w]=w$, then
  \[\lim_{n\to\yy}\max_{x\in[ct_n,h(t_n)]}\left[|u(t_n,x)-\varphi_{c_0}(h(t_n)-x)|+|v(t_n,x)-\psi_{c_0}(h(t_n)-x)|\right]=0, ~ ~ \forall c\in(0,c_0).\]
  \item When $\mathbb{B}[w]=w_x$, then
  \[\lim_{n\to\yy}\max_{x\in[0,h(t_n)]}\left[|u(t_n,x)-\varphi_{c_0}(h(t_n)-x)|+|v(t_n,x)-\psi_{c_0}(h(t_n)-x)|\right]=0.\]
\end{enumerate}

Due to Lemmas \ref{l4.1} and \ref{l4.2}, we know $h(t+t_n)-c_0(t+t_n)+2\tilde C\to L(0)$ in $C^{1+\frac{1+\alpha}{2}}_{{\rm loc}}(\mathbb{R})$ as $n\to\yy$, which implies that $h'(t+t_n)\to c_0$ in $C^{\frac{1+\alpha}{2}}_{{\rm loc}}(\mathbb{R})$ as $n\to\yy$. Thus the assertion (i) holds.

By Lemmas \ref{l4.1} and \ref{l4.2} again, we see that as $n\to\yy$,
\[(u(t+t_n,x+h(t+t_n)), v(t+t_n,x+h(t+t_n)))\to(\varphi_{c_0}(-x),\psi_{c_0}(-x)) ~ {\rm in ~ }[C^{1+\frac{\alpha}{2},2+\alpha}_{{\rm loc}}(\mathbb{R}\times(-\yy,0])]^2,\]
which indicates that for any $K_0>0$,
\bes\label{4.14}\lim_{n\to\yy}\|(u(t_n,x),v(t_n,x))-(\varphi_{c_0}(h(t_n)-x),\psi_{c_0}(h(t_n)-x))\|_{L^{\yy}([h(t_n)-K_0,h(t_n)])}=0.\ees
Thanks to Lemmas \ref{l3.2} and \ref{l3.3}, for any small $\ep>0$, there exist $K_1>0$ and large $N$ such that
\[(u^*-\ep,v^*-\ep)\le (u,v)(t_n,x)\le (u^*+\ep,v^*+\ep),\]
for $n\ge N$, $x\in[ct_n,h(t_n)-K_1]$ when $\mathbb{B}[w]=w$ and $x\in[0,h(t_n)-K_1]$ when $\mathbb{B}[w]=w_x$. Moreover, for large $K_2>0$, it is easy to see that
\[(u^*-\ep,v^*-\ep)\le (\varphi_{c_0},\psi_{c_0})(h(t_n)-x)\le (u^*+\ep,v^*+\ep) ~ ~ {\rm for ~ }x\le h(t_n)-K_2.\]
Thus we can choose $K_0=\max\{K_1,K_2\}$, then for $n\ge N$,
\[|(u(t_n,x),v(t_n,x))-(\varphi_{c_0}(h(t_n)-x),\psi_{c_0}(h(t_n)-x))|\le 2\ep,\]
uniformly in $[ct_n,h(t_n)-K_0]$ when $\mathbb{B}[w]=w$, and in $[0,h(t_n)-K_0]$ when $\mathbb{B}[w]=w_x$,
which, combined with \eqref{4.14}, leads to the assertions (ii) and (iii). Step 1 is ended.

{\bf Step 2.} In this step, we will show $h(t)-c_0t\to L(0)-2\tilde C$ as $t\to\yy$.

For the sequence $\{t_n\}$ satisfying $\lim_{n\to\yy}l_n(t)=\liminf_{t\to\yy}l(t)=C^*$, by Lemma \ref{l4.1} and Step 1, we know that there exists a subsequence, still denoted by itself, such that
\bes\label{4.15}
\lim_{n\to\yy}(h(t_n)-c_0t_n)=\liminf_{t\to\yy}(h(t)-c_0t)=L(0)-2\tilde C, ~ \lim_{n\to\yy}h'(t_n)=c_0,
\ees
and the assertions (ii) and (iii) hold.

Argue on the contrary that $h(t)-c_0t$ does not converges to $L(0)-2\tilde C$ as $t\to\yy$. Then there must exist a sequence $\{s_n\}$ with $s_n\to\yy$ such that
\[\lim_{n\to\yy}(h(s_n)-c_0s_n)=\limsup_{t\to\yy}(h(t)-c_0t)=:\bar{h}^*>L(0)-2\tilde C.\]

We now use the upper solution $(\bar{u},\bar{v},\bar{h})$ defined in Lemma \ref{l3.1} to derive a contradiction. Set $X_0=\sigma=(\bar{h}^*-L(0)+2\tilde C)/4>0$ and $T=t_n$. Then we can choose $\delta$ and $K$ suitably such that \eqref{3.3} is valid. Thus to prove \eqref{3.1}, it remains to verify that $(\bar{u}(t_n,x),\bar{v}(t_n,x))\ge (u(t_n,x), v(t_n,x))$ for $x\in[0,h(t_n)]$.

In view of the assertion (iii) in Step 1, we have that for $x\in[0,h(t_n)]$,
\bess
(\bar{u}(t_n,x),\bar{v}(t_n,x))&=&(1+K)(\varphi_{c_0}(\bar{h}(t_n)-x),\psi_{c_0}(\bar{h}(t_n)-x))\\
&=&(1+K)(\varphi_{c_0}(h(t_n)+X_0-x),\psi_{c_0}(h(t_n)+X_0-x))\\
&\ge& (u(t_n,x), v(t_n,x))
\eess
provided that $n$ is large enough. Using the assertion (ii), we similarly can derive that if $n$ is sufficiently large, then for $x\in[ct_n,h(t_n)]$,
\bess
(\bar{u}(t_n,x),\bar{v}(t_n,x))\ge(u(t_n,x), v(t_n,x))
\eess
As for $x\in[0,ct_n]$, recall that $c<c_0$ and
\bess
\limsup_{t\to\yy}(u(t,x),v(t,x))\le (u^*,v^*) ~ ~ {\rm uniformly ~ in ~ }[0,\yy).
\eess
Thus for $x\in[0,ct_n]$,
\bess
(\bar{u}(t_n,x),\bar{v}(t_n,x))&=&(1+K)(\varphi_{c_0}(\bar{h}(t_n)-x),\psi_{c_0}(\bar{h}(t_n)-x))\\
&=&(1+K)(\varphi_{c_0}(h(t_n)+X_0-x),\psi_{c_0}(h(t_n)+X_0-x))\\
&\ge&(1+K)(\varphi_{c_0}(h(t_n)+X_0-ct_n),\psi_{c_0}(h(t_n)+X_0-ct_n))\\
&\ge&(u(t_n,x), v(t_n,x))
\eess
provided that $n$ is large enough. So $(\bar{u}(t_n,x),\bar{v}(t_n,x))\ge (u(t_n,x), v(t_n,x))$ for $x\in[0,h(t_n)]$ and all large $n$. Then by Lemma \ref{l3.1}, we obtain that $h(t)\le \bar{h}(t)$ for $t\ge t_n$ and all large $n$, which implies that for large $s_k$ satisfying $s_k\ge t_n$, we have
\[h(s_k)\le \bar{h}(s_k)=c_0(s_k-t_n)+\sigma(1-e^{-\delta(s_k-t_n)})+h(t_n)+X_0.\]
Furthermore,
\[\bar{h}^*=\lim_{k\to\yy}(h(s_k)-c_0s_k)\le-c_0t_n+\sigma+h(t_n)+X_0.\]
Letting $n\to\yy$ leads to
\[\bar{h}^*\le L(0)-2\tilde C+\sigma+X_0=L(0)-2\tilde C+(\bar{h}^*-L(0)+2\tilde C)/2,\]
which is a contradiction. Thus Step 2 is finished.

With the aid of the two steps, we can use any positive sequence $\{t_n\}$ with $t_n\to\yy$ to argue as in Lemma \ref{l4.1}, and derive a subsequence such that the assertions (ii) and (iii) in Step 1 and \eqref{4.15} hold, which obviously completes the proof.
\end{proof}

\end{document}